\documentclass[11pt]{article}
\usepackage{amssymb}
\usepackage{mathrsfs}
\usepackage{latexsym,amsmath,amssymb,amsfonts,epsfig,graphicx,cite,psfrag}
\usepackage{eepic,color,colordvi,amscd}
\usepackage{ebezier}
\usepackage{verbatim}
\usepackage{subcaption}
\usepackage{enumitem}
\usepackage{algorithm}
\usepackage{algpseudocode}
\usepackage[normalem]{ulem}
\usepackage{hyperref}
\usepackage{amsthm}
\usepackage{longtable}
\usepackage{comment}
\usepackage{color}
\usepackage{xcolor}
\usepackage{graphicx}
\usepackage{epstopdf}
\usepackage{longtable}
\usepackage{booktabs}
\usepackage{mathtools}
\usepackage{multicol, multirow}
\usepackage{capt-of}
\usepackage{longtable}
\usepackage{amsopn}
\usepackage{tikz}
\usepackage{extarrows}
\usetikzlibrary{positioning,arrows.meta,decorations.pathreplacing}
\usepackage{float}
\definecolor{blue}{rgb}{0,0,0.9}
\definecolor{red}{rgb}{0.9,0,0}
\definecolor{green}{rgb}{0,0.9,0}

\newcommand{\cU}{{\cal U}}

\newcommand{\cN}{{\cal N}}

\newcommand{\Rbar}{\overline{\mathbb{R}}}

\theoremstyle{plain}
\newtheorem{remark}{Remark}
\newtheorem{example}{Example}[section]

\newtheorem{assumption}{Assumption}
\newtheorem{definition}{Definition}
\newtheorem{proposition}{Proposition}
\newtheorem{theorem}{Theorem}
\newtheorem{corollary}{Corollary}

\def\<{\big\langle}
\def\>{\big\rangle}

\def\T{{\cal T}}
\DeclareMathOperator{\dom}{dom}

\def\bX{{\mathbb{X}}}
\def\bY{{\mathbb{Y}}}

\def\K{\mathcal{K}}
\def\F{\mathcal{F}}

\def\R{\mathbb{R}}

\def\prox{{\operatorname{Prox}}}
\def\Prox{{\operatorname{Prox}}}

\def\N{\mathcal{N}}

\def\S{\mathbb{S}}

\def\dist{{\rm dist}}
\def\dd{{\rm diag}}

\DeclareMathOperator{\epi}{epi}

\def\lin{{\operatorname{lin}}}

\def\bB{{\mathbb{B}}}

\def\tol{\tt{tol}}
\def\timelimit{\tt{TimeLimit}}
\def\iterlimit{\tt{IterLimit}}

\def\co{{\rm co}}

\def\env#1{E_{#1}}

\DeclareMathOperator{\diag}{diag}

\let\svthefootnote\thefootnote
\newcommand\blankfootnote[1]{
\let\thefootnote\svthefootnote
}
\usepackage[many]{tcolorbox}
\newtcolorbox{boxA}{
    fontupper = \bf,
    boxrule = 1.5pt,
    colframe = black 
}

\newcommand{\xdownarrow}[1]{%
  {\left\downarrow\vbox to #1{}\right.\kern-\nulldelimiterspace}
}

\def\dd{\Phi_z}
\def\E{\mathcal{E}}
\def\fE{f_{\E}}

\begin{document}
        \title{
        On the efficient computation of proximal operators of affine-constrained nonconvex functions
        }
        \author{
        Di Hou\thanks{Department of Mathematics, National
          University of Singapore, Singapore
          119076 ({\tt dihou@u.nus.edu}).
          }, \quad 
	    Tianyun Tang\thanks{Department of Statistics, The University of Chicago ({\tt ttang@u.nus.edu}).
          }, \quad 
	    Kim-Chuan Toh\thanks{Department of Mathematics, and 
          Institute of Operations Research and Analytics, 
          National University of Singapore, 
          Singapore
          119076 ({\tt mattohkc@nus.edu.sg}). 
         }, \quad
        Shiwei Wang\thanks{
Institute of Applied Mathematics, Academy of Mathematics and Systems Science, Chinese Academy of Sciences, Beijing,
China. This work is done while the author is a visiting scholar at the Institute of Operational Research and Analytics, National
University of Singapore, Singapore ({\tt wangshiwei@amss.ac.cn}).
          }
 	}
	\date{\today}
\maketitle


\begin{abstract}

Proximal operators with affine constraints arise in numerous models in nonconvex projection, composite optimization, and structured regularization. However, their efficient computation remains challenging due to the simultaneous presence of affine constraints and nonsmooth, possibly nonconvex objectives. In this work, we develop a unified dual-representability framework for analyzing and computing affine-constrained proximal mappings. 
Specifically, we introduce a multiplier inclusion formulation that connects the primal affine-constrained proximal problem to an unconstrained convex dual problem. Based on this formulation, we prove that, whenever the associated dual inclusion problem admits a solution, strong duality holds. 
For convex functions and a broad class of prox-regular nonconvex functions,  we establish that dual representability holds under a simple subdifferential sum rule, and further develop a hierarchy of verifiable regularity conditions that guarantee this sum rule. 
In addition, we analyze the smoothness and strong convexity properties of the dual objective, providing a rigorous foundation that guarantees fast local convergence rates for efficient first- and second-order methods. 
Numerical experiments demonstrate that the proposed dual reformulation enables the reliable computation of globally optimal solutions for a range of large-scale nonconvex proximal and projection problems using existing convex optimization solvers. 

\end{abstract}

 \noindent{\bf keywords:} Strong duality, Nonconvex  proximal mapping, Dual representability


\section{Introduction}

In this paper, we 
consider the function $\fE: \bX\mapsto \R\cup\{+\infty\}$ defined by
\begin{equation}\label{deffE}
    \fE(x)\coloneqq f(x)+\delta_{\E}(x),
\end{equation}
where $f:\bX \to\R\cup\{+\infty\}$ is a proper, lower semicontinuous (lsc), but possibly nonconvex function, and $\delta_{\E}(x)$ is the indicator function of an affine set ${\E}\coloneqq \{x\in\bX\;|\; Ax=b\}$, which takes the value $0$ if $x\in \E$, and $+\infty $ otherwise. The spaces $\bX$ and $\bY$ are finite-dimensional Hilbert spaces, $A:\bX \to\bY $ is a surjective linear operator and $b\in\bY $ is a fixed vector. We assume that the effective domain of $\fE$, given by 
\begin{equation}\label{Edom}
\mathcal{F}\coloneqq \E\cap\dom f,
\end{equation}
is nonempty. Under these assumptions, we study the following proximal minimization problem associated with $\lambda\fE$:
\begin{equation}\label{Prox}
    \env{\lambda \fE}(z)\coloneqq\min_{x\in\bX }\ \left\{\lambda f(x)+\frac{1}{2}\|x - z\|^2
    \;\mid\; A x = b\right\},\tag{P}
\end{equation}
where $\lambda>0$ is a  parameter.
Proximal minimization problems of the form~\eqref{Prox} arise as the subproblems in many important optimization algorithms, including proximal gradient methods, proximal point algorithms, and operator-splitting schemes, see, e.g., \cite{parikh2014proximal,rockafellar1976monotone}. In these methods, each iteration requires the evaluation of a proximal mapping, and the overall performance depends critically on how accurately and efficiently this subproblem can be solved.
A particularly important special case is the projection onto a feasible set under affine constraints, which takes the form
\begin{equation}\label{Proj}
\Pi_{\K\cap \E}(z)=\underset{x\in\bX}{\arg\min}\ \left\{\delta_\K(x)+\frac{1}{2}\|x - z\|^2 \;\middle|\; A x = b\right\},
\end{equation}
where $\mathcal K\subseteq\bX $ is a closed, possibly nonconvex set.
Projection problems of the form~\eqref{Proj} arise in a wide range of applications, where the constraint set $\K$ incorporates important structural information, such as sparsity~\cite{deng2025alternating}, low-rank structure~\cite{li2025b}, and geometric properties~\cite{RiNNAL,manibook}. While~\eqref{Proj} is conceptually simple, it is already computationally challenging due to the simultaneous presence of rigid affine constraints and nonsmooth, possibly nonconvex geometry.
Consequently, most  existing literature focuses on finding a feasible point instead of the exact projection, typically via first-order, alternating, or relaxation-based methods~\cite{xiao2025quadratically,xiao2025exact,lewis2008alternating,drusvyatskiy2015transversality}.

In this work, instead of finding approximate or local optimal solutions, we show that exact affine-constrained proximal and projection mappings can be computed under verifiable structural conditions in both convex and nonconvex settings through a unified dual inclusion framework that delivers both theoretical guarantees and efficient algorithms.


\subsection{Dual representation}

The difficulty of solving the primal problem~\eqref{Prox} motivates the search for alternative formulations that can better reveal the underlying structure of the proximal problem and enable more efficient computation. In Section~\ref{sec:dual}, we investigate the convex dual formulation of~\eqref{Prox}, which transforms the original constrained minimization problem into an inclusion problem in the dual space:
\begin{equation}\label{eq:inclusion}
    0\in F(y) := A\,\prox_{\lambda f}(z + A^* y) - b,
\end{equation}
where the proximal operator $\Prox_{\lambda f}$ is the set-valued mapping defined by
\begin{equation}\label{eq:prox}
\prox_{\lambda  f}(t)\coloneqq\underset{x\in \bX }{\arg\min}\;\left\{\lambda f(x)+\frac{1}{2}\|x-t\|^2\right\}.
\end{equation}
It is well known that for a general optimization problem, 
its corresponding dual problem may be infeasible. Even when a solution exists, solving the dual problem can be highly nontrivial. More fundamentally, for general optimization problems, a nonzero duality gap may arise, so that dual solvability alone does not guarantee primal solvability.

However, due to the special structure of the primal problem~\eqref{Prox}, we will prove in Proposition~\ref{prop:SD-eq} that as long as its dual inclusion problem~\eqref{eq:inclusion} admits a solution $y^*$, then there is no duality gap, and an optimal solution to the primal problem can be recovered directly via
\begin{equation*}
x^* \in \Prox_{\lambda f}(z + A^*y^*).
\end{equation*}
Thus, to find a global solution of~\eqref{Prox}, it suffices to find a solution to~\eqref{eq:inclusion}. We say that~\eqref{Prox} is dual-representable if such a solution exists.

We provide a complete characterization of the dual representability of~\eqref{Prox}. If the function $f$ is convex, then the operator $F:\bY\to\bY$ defined in~\eqref{eq:inclusion} is single-valued and monotone, since the proximal mapping $\Prox_{\lambda f}$ is single-valued and firmly nonexpansive. In Subsection~\ref{subsec:cvx-prox}, we prove that, in this convex case, ~\eqref{Prox} is dual-representable if and only if the subdifferential sum rule (SSR)
\[
\partial (f_{\E})(x)=\partial f(x)+ A^* \bY
\]
holds, where $f_\E$ is defined in \eqref{deffE}. To ensure the validity of the SSR, we further establish a hierarchy of regularity conditions based on commonly used and verifiable constraint qualifications. This unifies and extends classical results \cite{deutsch1997dual} from convex projection problems to the more general convex proximal setting.
The nonconvex case is more complicated because the proximal mapping $\Prox_{\lambda f}$ may become set-valued and discontinuous, so that the operator $F$ in \eqref{eq:inclusion} is generally nonmonotone and possibly multi-valued. In Subsection~\ref{subsec:nonconvex-prox}, under mild prox-regularity and a nonconvex version of the SSR, we derive sufficient conditions for dual representability and again provide a verifiable regularity hierarchy. This extends the analysis of special affine-constrained 
spherical sets in~\cite{RiNNAL} to general affine-constrained proximal problems.

To further enhance the applicability of our framework, we investigate the regularity properties of the dual problem, with a particular focus on smoothness and strong convexity. We establish differentiability of the dual objective, together with explicit subdifferential and gradient characterizations, in Subsection~\ref{subsec:property-dual}. Building on these results, we characterize the strong convexity via the strong monotonicity of an associated operator in Subsection~\ref{sec:scvx}, and apply it to the $C^2$ smooth manifold projection problem in Subsection~\ref{subsec-C2-mani}.
This analysis generalizes the strong convexity theory for projections onto spherical manifolds developed in~\cite[Proposition~1]{RiNNAL} to projection problems over general smooth manifolds.

In Section~\ref{sec:applications}, we apply the proposed framework to a broad class of problems, including projection, proximal gradient, and proximal composite models, and show that each admits a solvable dual inclusion formulation. This demonstrates that a wide range of affine-constrained proximal problems, both convex and nonconvex, can be treated within a single dual-representability framework.
For several representative problem classes, the framework further yields sharper results than those available in existing analyses. In particular, for projection-type problems, it leads to clearer dual characterizations and stronger regularity conclusions under weaker assumptions.
More generally, rather than relying on problem-specific derivations, the framework provides a unified approach for deriving dual representations and structural properties. Many known formulas and regularity results can be recovered in a systematic manner by verifying a common set of assumptions, while the same approach also facilitates the derivation of new results with significantly reduced technical overhead.


\subsection{Contributions}

The main contributions of this paper are summarized as follows:

\begin{enumerate}
\item We develop a unified multiplier inclusion framework that reformulates the affine-constrained proximal problem \eqref{Prox} in the dual space. We prove that, whenever the associated dual inclusion admits a solution, strong duality holds with zero duality gap, and every dual solution yields a primal solution via a proximal evaluation. We further characterize the equivalent conditions under which this exact dual representation holds, thereby identifying when affine-constrained nonconvex proximal problems can be solved through its convex dual formulations.

\item In the convex setting, we show that strong duality and the existence of multiplier solutions for~\eqref{Prox} are equivalent to the validity of a simple subdifferential sum rule (SSR) for the pair $(f,\delta_\E)$.
To make this condition verifiable in practice, we then develop a hierarchy of regularity conditions that comprises several commonly used and easily checkable constraint qualifications, each of which guarantees SSR. This hierarchy extends the classical convex projection theory: when $f=\delta_{\cal K}$, it reduces to the usual chain of conditions culminating in the strong conical hull intersection property
(CHIP) for the pair $({\cal K},\E)$.

\item In the nonconvex setting, we show that local dual representability of~\eqref{Prox} is ensured under mild prox-regularity and a nonconvex subdifferential sum rule (NSSR) for the pair $(f,\delta_\E)$. Moreover, we also establish a hierarchy of verifiable regularity conditions that guarantee NSSR. Under these conditions, the original nonconvex proximal mapping can be characterized via a convex dual inclusion, whose solution yields a globally optimal primal solution. This result substantially extends the special retraction-based theory in RiNNAL~\cite{RiNNAL} to general prox-regular functions and provides a principled convex-analysis pathway for solving nonconvex projection problems.

\item We analyze the smoothness and strong convexity properties of the dual objective induced by the multiplier inclusion formulation. In particular, we establish continuity and differentiability results under mild regularity conditions, and derive explicit expressions for the dual gradient and subdifferential. 
We provide equivalent characterizations of strong convexity in terms of the strong monotonicity of an associated operator, and further demonstrate strong convexity of the dual objective function for projection problems onto $C^2$ manifold under nondegeneracy conditions.
These properties provide the theoretical foundation for efficient first- and second-order methods for solving the dual problem.

\item We demonstrate that the proposed dual representability framework unifies a broad class of projection and proximal type problems under a single analytical framework. By applying the framework to projection, proximal gradient, and proximal composite models, we show that many existing problem-specific dual formulations and regularity results can be systematically recovered by verifying appropriate assumptions, while substantially simplifying the derivation of optimality conditions. In particular, for projection problems, the framework yields clearer dual representations and sharper regularity and strong convexity characterizations.

\item We conduct extensive numerical experiments to demonstrate that the proposed dual reformulation enables the effective computation of solutions to affine-constrained nonconvex proximal problems by applying existing solvers 
to the resulting convex unconstrained dual problems. The results confirm that this approach reliably delivers globally optimal primal solutions and remains efficient on large-scale instances.

\end{enumerate}

Together, these results show that a broad class of affine-constrained nonconvex proximal problems admits an exact dual representation in terms of convex unconstrained problems, thereby unifying projection-type and proximal-type formulations under a common multiplier framework and enabling the computation of global solutions of these challenging nonconvex problems via well established convex optimization techniques.


\subsection{Organization}

The rest of the paper is organized as follows. 
Section~\ref{sec:dual} introduces the dual representation of the proximal problem, analyzes the smoothness and strong convexity properties of the associated dual function, and characterizes dual representability (zero duality gap) in both the convex and nonconvex settings. Section~\ref{sec:applications} presents applications of the proposed framework. Numerical experiments are reported in Section~\ref{sec:numerical}, followed by concluding remarks in Section~\ref{sec:conclusion}.


\subsection{Notation}

We use $\Rbar:= \R \cup \{\pm\infty\}$ to denote the extended real line, $\mathbb{B}_{\epsilon}(\bar{x}) := \{x \in \mathbb{R}^n \mid \|x - \bar{x}\| < \epsilon\}$ to denote the open ball centered at $\bar{x}$ with radius $\epsilon > 0$. We let $[n] \coloneqq \{1,2,...,n\}$
for any positive integer $n$. 
For a set ${\cal S}\subset\mathbb{X}$, $\co ({\cal S})$ denote its convex hull, and $\lin({\cal S})$ denotes the largest linear subspace contained in ${\cal S}$. 
For a function $f:\mathbb{X}\to\overline{\mathbb{R}}$, we denote by
${\rm d}f$ the subderivative,
$\widehat{\partial}f$ the regular subgradient,
$\partial f$ the limiting subgradient, 
and $\partial^\infty f$ the horizon subgradient, 
as defined in~\cite[Definition~8.1, 8.3]{rockafellar1998variational}.
When \(f\) is the indicator function of a set \({\cal S}\), the regular and limiting subgradients \(\widehat{\partial} f\) and \(\partial f\) reduce to the corresponding regular and limiting normal cones \(\widehat{\cal N}_{\cal S}\) and \({\cal N}_{\cal S}\), respectively (see Definition~\ref{def:normal-cone} below).
We denote the conjugate of $f$ by $f^*$,
and define its (effective) domain as
\[
\dom f \;:=\; \{x\in\mathbb{X}\;|\; f(x)<+\infty\}.
\]
The epigraph of 
$f$ is defined by
\[
\epi f\;\coloneqq\; \{(x,\tau)\in\bX\times \R \;\mid\; \tau\geq f(x),\; x\in \dom f\}.
\]
The {proximal mapping} of \(f\) is the set-valued operator
\[
\Prox_{f}(t)
\;:=\;
\underset{x\in\mathbb{X}}\arg\min
\Big\{
f(x) + \frac{1}{2}\|x - t\|^{2}
\Big\},
\quad t\in\mathbb{X}.
\]
The corresponding Moreau envelope is
\[
E_{f}(t)
\;:=\;
\min_{x\in\mathbb{X}}
\Big\{
f(x) + \frac{1}{2}\|x - t\|^{2}
\Big\},
\quad t\in\mathbb{X}.
\]
For a set-valued mapping $\mathcal{V}:\mathbb{X}\rightrightarrows\mathbb{Y}$, its outer limit at $\bar{x}$ is defined by 
$$\limsup\limits_{x\rightarrow\bar{x}}\mathcal{V}(x):=\left\{d\in\mathbb{Y} \mid \exists\;x_k\rightarrow\bar{x},\; d_k\rightarrow d\;{\rm such\;that}\;d_k\in\mathcal{V}(x_k)\;\forall\,k\right\}.$$
The following definitions of the normal and tangent cones of sets are taken from \cite[Definition 6.3, Proposition 6.5, and Example 6.16]{rockafellar1998variational}. 
\begin{definition}\label{def:normal-cone}
Let  ${\cal S}$ be a nonempty closed subset of $\mathbb{X}$ and  $\bar{x}\in {\cal S}$ be given. We call
	\begin{equation}\label{def:regular-normal-cone}
		\widehat{\cal N}_{\cal S} (\bar{x}):= \{d\in\mathbb{X}\mid \langle {d}, x-\bar{ x}\rangle \leq o(\|{x}-\bar{x}\|) \quad \forall\, {x}\in {\cal S}\}
	\end{equation}
	the regular normal cone to set ${\cal S}$ at point $\bar{x}$ and 
\begin{equation*}\label{eq:def-limit-norm}
		{\cal N}_{\cal S} (\bar{x}):= \limsup\limits_{x\rightarrow\bar{x}}\widehat{\cal N}_{\cal S} ({x})
	\end{equation*}
	the limiting normal cone
    to set ${\cal S}$ at point $\bar{x}$. 
\end{definition}

\begin{definition}
Let  ${\cal S}$ be a nonempty closed subset of $\mathbb{X}$ and  $\bar{x}\in {\cal S}$ be given. 
We call
$${\cal T}_{\cal S}(\bar{x})=\{d\mid{\rm dist}(\bar{x}+td,{\cal S})=o(t),\;t\geq0\}$$
the tangent cone to set ${\cal S}$ at point $\bar{x}\in{\cal S}$.
\end{definition}


\section{Dual representation}\label{sec:dual}

In this section, we summarize key properties of the affine-constrained proximal problem~\eqref{Prox} and derive its dual formulation. Both convex and nonconvex settings are analyzed, and conditions are established under which strong duality holds and the dual formulation admits an optimal solution. We begin by recalling the notion of prox-boundedness from~\cite[Definition~1.23]{rockafellar1998variational}, which requires that \(f\) majorizes a quadratic function and hence ensures that its Moreau envelope is well-defined.

\begin{definition}[\bf Prox-boundedness]
A function $f:\bX \to \overline{\mathbb{R}}$ is said to be {prox-bounded} if there exists $\gamma>0$ such that the Moreau envelope $E_{\gamma f }(x) > -\infty$ for some $x\in\bX$. The supremum $\lambda_f$ of the set of all such $\gamma$ is called the {threshold} of prox-boundedness for $f$.
\end{definition}

Next we introduce a standing assumption that guarantees the well-posedness of both the primal problem~\eqref{Prox} and its dual formulation~\eqref{D}, where $\bX$ and $\bY$ are finite-dimensional Hilbert spaces.

\begin{assumption}\label{assump:basic}
The function \(f:\bX\to\Rbar\) is proper, lsc, and prox-bounded with threshold $\lambda_f>\lambda$.  
The linear operator \(A:\bX\to\bY\) is surjective, and $\mathcal{F}$ defined in \eqref{Edom} is nonempty.
\end{assumption}

We now derive the dual formulation of problem~\eqref{Prox}. The Lagrangian function associated with problem~\eqref{Prox} is
\[
   L(x,y) = \lambda f(x)+\frac{1}{2} \|x - z\|^2 + \langle y, b - A x \rangle.
\]
Denote $v(y):=z+A^*y$. Then the corresponding dual function is given by
\begin{equation}\label{eq:dual-function-prox}
    \Big(\inf_{x \in \bX }\; L(x,y) \Big)  \; = \env{\lambda f}(v(y))- \frac{1}{2} \|v(y)\|^2+\langle b,y \rangle + \frac{1}{2} \|z\|^2
    \;=:\; -\dd(y),
\end{equation}
where \(E_{\lambda f}\) denotes the Moreau envelope of \(\lambda f\),
and the minimizer of the above minimization problem satisfies \(x\in\mathrm{Prox}_{\lambda f}(v(y))\).
The dual problem associated with~\eqref{Prox} can be written in the following form:
\begin{equation}\label{D}
    v(D) = -\min_{y \in \bY }\; \dd(y).\tag{D}
\end{equation}
Since our goal is to solve the primal problem~\eqref{Prox} via its dual formulation~\eqref{D}, it is essential to characterize conditions under which the two problems are equivalent. The following proposition provides the equivalent characterization of strong duality.

\begin{proposition}\label{prop:SD-eq} Under Assumption~\ref{assump:basic},
the following statements are equivalent:
\begin{enumerate}
[label=\bf (\roman*)]
\item \emph{\bf  (Attainable strong duality)} there exist $x^*\in \F$ and $ y^*\in \bY$ such that
\begin{equation*}\label{eq:equal-values}
  v(P) 
  = 
  \lambda f(x^*)+\frac{1}{2} \|x^*-z\|^2 
  = 
  -\dd(y^*) 
  =
  v(D),
\end{equation*}
i.e., $x^*$ solves \eqref{Prox} and $ y^*$ solves \eqref{D}, and there is no duality gap.
\item \emph{\bf (Lagrangian saddle point)} there exist $x^*\in \F$ and $ y^*\in \bY$ such that
\begin{equation}\label{eq:saddle-system}
  x^* \in \underset{x\in\bX }{\arg\min}\; L(x, y^*) 
  .
\end{equation}
\item \emph{\bf (Dual inclusion)} there exists $ y^*\in \bY$ such that
\begin{equation}\label{eq:dual-repre}
    0\in F(y^*) = A\,\prox_{\lambda f}(z + A^* y^*) - b.
\end{equation}
\end{enumerate}
\end{proposition}

\begin{proof}
{(i) $\Leftrightarrow$ (ii).}
The equivalence follows from standard saddle-point duality, see \cite[Theorem~2.158]{bonnans2013perturbation}.

{(ii) $\Rightarrow$ (iii).}
Fix $x^*$ and $y^*$ in (ii). The optimality condition for the $x$–minimization in $L(\cdot,y^*)$ indicates that 
$
x^* \in \prox_{\lambda f}(z + A^{*}y^*).
$
To verify~{(iii)},
we use the feasibility condition $A x^*=b$ (since $x^*\in\mathcal F$) to obtain that
\[
  0= A x^* - b \;\in \; A\,\prox_{\lambda f}(z + A^{*}y^*) - b.
\]

{(iii) $\Rightarrow$ (ii).}
Assume that $0\in F(y^*)$. Then by (iii), there exists 
$
x^* \in \prox_{\lambda f}(z + A^{*}y^*)
$
such that $Ax^*=b$, and hence $x^*\in\mathcal F$. By the definition of proximal mapping, $x^*$ solves
\[
  x^* \in \underset{x\in\bX }{\arg\min}\;\Bigl\{\lambda f(x)+\tfrac12\|x-z\|^2 - \langle A^{*}y^*,x\rangle\Bigr\}
  \;=\; \underset{x\in\bX }{\arg\min}\; L(x,y^*),
\]
which is exactly (ii).
\end{proof}

\begin{remark}
    The above equivalence is useful because it yields a computationally tractable characterization of primal 
    optimality in the dual space. For projection problems, such dual characterizations are well established in the convex setting (see, e.g., \cite{deutsch1997dual}), but for nonconvex projections, they remain far less developed and are only partially treated in a few special cases (e.g., \cite{qi2014computing,RiNNAL}).
\end{remark}

\begin{definition}[\bf Dual-representable]
We say that problem~\eqref{Prox} is {dual-representable} at $z\in\bX$ if any (and hence all) of the equivalent conditions in Proposition~\ref{prop:SD-eq} are satisfied.
\end{definition}

In the following subsections, we first provide some basic properties of the dual function, and then characterize the conditions under which problem~\eqref{Prox} is dual-representable, thereby ensuring that  the primal problem~\eqref{Prox} can be  solved through its unconstrained convex dual problem~\eqref{D}.


\subsection{Properties of the dual function}\label{subsec:property-dual}

In this subsection, we analyze basic regularity properties of the dual objective function $\dd(y)$. These properties provide a theoretical foundation for the design and analysis of efficient algorithms for solving the dual problem, together with explicit characterizations of its subgradients and gradients.

\begin{proposition}\label{prop:basic-props-dual}
Under Assumption~\ref{assump:basic}, the dual objective function $\dd$ satisfies:
\begin{enumerate}
[label=\bf (\roman*)]
    \item $\dd$ is continuous, convex and locally Lipschitz on $\bY $.
    \item The subdifferential of $\dd$ at $y$ is given by
    \begin{equation}\label{eq:subdiff-co}
        \partial \dd (y) = \operatorname{co}\!\big(A\,\prox_{\lambda f}(z + A^*y) - b\big).
    \end{equation}
    \item If the proximal solution is unique, then $\dd$ is differentiable at $y$ with
    \begin{equation}\label{eq:grad-dual}
        \nabla \dd (y) = A\,\prox_{\lambda f}(z + A^*y) - b.
    \end{equation}
\end{enumerate}
\end{proposition}

\begin{proof}
Recall that $v(y):=z+A^*y$ and
\[
L(x,y):=\lambda f(x)+\tfrac12\|x-z\|^2+\langle y,\,b-Ax\rangle,\quad
\dd(y)=\sup_{x\in\bX }\; \{-L(x,y)\}.
\]
\emph{(i) Continuous, convex and locally Lipschitz.}
For each fixed $x \in \bX $ with $f(x) < +\infty$, the map $y \mapsto -L(x,y)$ is affine (hence continuous). If $f(x) = +\infty$, then $-L(x,y) \equiv -\infty$, and such $x$ do not affect the supremum. Thus, $\dd$ is the pointwise supremum of continuous affine functions and is therefore convex and lsc.
By Assumption~\ref{assump:basic}, the function $f$ is proper, lsc, and prox-bounded with threshold $\lambda_f>\lambda$. Hence, by~\cite[Theorem~1.25]{rockafellar1998variational}, the Moreau envelope $E_{\lambda f}$ is finite everywhere, and consequently $\dd(y)$ is finite for all $y\in\bY$.
A convex function that is finite everywhere on a Hilbert space is continuous (indeed, locally Lipschitz) on $\bY $.

\emph{(ii) Subdifferential formula.}
For each fixed $x\in\dom f $, the mapping $y\mapsto L(x,y)$ is affine, hence
continuously differentiable with $\nabla_y L(x,y)=b-Ax$.
For each fixed $y\in\bY $, the mapping $x\mapsto L(x,y)$ is proper and closed. Moreover,
by~\cite[Theorem~1.25]{rockafellar1998variational}, the inner solution set of $\dd(y)= -\inf_x L(x,y)$ is nonempty and locally uniformly compact for all $y\in\bY$.
Therefore we can apply the subdifferentiation rule for parametric minimization~\cite[Proposition A.22]{bertsekas1971control} to 
$\dd(y)=-\inf_x L(x,y)$,
obtaining
\begin{equation}\label{eq:DText}
\partial \dd(y)
= \operatorname{co}\Big\{-\nabla_{ y}L(x,y)\mid 
x \in \operatorname*{arg\,min}_{x} L(x,y)\Big\}
= \operatorname{co}(A \Prox_{\lambda f}(x+A^* y) - b).
\end{equation}

\emph{(iii) Differentiability under a unique proximal point.}
If $\Prox_{\lambda f}(v(y))$ is a singleton, then the convex set $\partial\dd(y)$ is a singleton, hence
$\dd$ is differentiable at $y$, consequently, \eqref{eq:subdiff-co} simplifies to \eqref{eq:grad-dual}.
\end{proof}

\begin{remark}

The application of \cite[Proposition~A.22]{bertsekas1971control} in the proof above formally assumes that $x$ lies in a compact set.  
However, inspecting the proof reveals that compactness is not essential: the argument still goes through as long as the minimizer set $\arg\min_x L(x,y)$ is nonempty and uniformly bounded.  
\end{remark}


\subsection{Convex proximal mapping}\label{subsec:cvx-prox}

In this subsection, we restrict our attention to the convex setting and characterize the necessary and sufficient conditions for dual-representability. These results form the basis for understanding the nonconvex case discussed in the next subsection. We make the following convex assumption throughout this subsection.

\begin{assumption}\label{assump:cvx}
    The function $f:\bX \to\Rbar$ is proper, closed and convex.
The linear operator $A:\bX \to\bY $ is surjective, 
and \(\mathcal{F}\neq\emptyset\).
\end{assumption}

The following rule provides the essential condition for establishing the dual representability of problem~\eqref{Prox}.

\begin{definition}[\bf Subdifferential sum rule]
Let \( f, g : \bX  \to \Rbar \) be proper, closed, and convex functions defined on a finite-dimensional Hilbert space $\bX$.  
The {subdifferential sum rule} (SSR) is said to hold for \( f \) and \( g \) at $x$ if $ x \in \dom f  \cap \dom g$ and
\[
\partial (f+g)(x) = \partial f(x) + \partial g(x).
\]
\end{definition}

We focus on the pair $(f, \delta_{\E})$. At any point $x\in \dom \fE=\F$, the subdifferential sum rule (SSR) takes the form
\[
\partial(\fE)(x)=\partial(f+\delta_{\E})(x)=\partial f(x) + {\cal N}_{\E}(x)=\partial f(x) + A^* \bY.
\]

The following proposition shows that, for a fixed $z$, the subdifferential sum rule (SSR) at an optimal solution is sufficient to guarantee dual-representability.

\begin{proposition}
\label{prop:SSR-dual-reprent}

Under Assumption~\ref{assump:cvx}, fix \(z\in\bX\) and let \(x^*\) be any solution of~\eqref{Prox}.  
If SSR holds for \(f\) and \(\delta_{\E}\) at \(x^*\), then~\eqref{Prox} is dual-representable at \(z\).

\end{proposition}

\begin{proof}
Since $\tfrac{1}{2\lambda}\|\cdot - z\|^2$ is everywhere finite and continuous, and $\fE$ is proper, closed, and convex, the subdifferential calculus (see \cite[Exercise~10.10]{rockafellar1998variational}) implies that
\[
\partial(\fE+\tfrac{1}{2\lambda}\|\cdot -z\|^2)(x^*)
= \partial \fE(x^*) + \tfrac{1}{\lambda}(x^*-z)=\partial f(x^*)+A^*\bY +\tfrac{1}{\lambda}(x^*-z),
\]
where the last equality uses SSR of $f$ and $\delta_\E$ at \(x^*\).
By the first-order necessary condition of~\eqref{Prox} at $x^*$, there exists $y^* \in \bY$ such that
\[
0 \in \partial f(x^*) -\tfrac{1}{\lambda}A^*y^*+ \tfrac{1}{\lambda}(x^*-z)
\quad\Longleftrightarrow\quad 
x^*\in \Prox_{\lambda f}(z+A^*y^*).
\]
Since $Ax^*=b$, we have $0\in A \Prox_{\lambda f}(z+A^*y^*)-b$, 
which  
by (iii) of Proposition~\ref{prop:SD-eq}, implies that~\eqref{Prox} is dual-representable.
\end{proof}

Proposition~\ref{prop:SSR-dual-reprent} shows that SSR is sufficient to guarantee 
dual-representability at a fixed~$z$. 
Our next result reveals that no weaker condition suffices: from a global perspective, this condition is also necessary.

\begin{theorem}[\bf Convex dual representability]\label{thm:strong-chip}
Under Assumption~\ref{assump:cvx}, the following statements are equivalent:
\begin{enumerate}
[label=\bf (\roman*)]
\item $f$ and $\delta_{\E}$ satisfies SSR for all $x\in\F $.
\item \eqref{Prox} is dual-representable for every $z \in \bX $.
\end{enumerate}
\end{theorem}

\begin{proof}

\smallskip
\noindent
{(i) $\Rightarrow$ (ii).}
It directly follows from Proposition~\ref{prop:SSR-dual-reprent}.

\smallskip
\noindent
{(i) $\Leftarrow$ (ii).}
Conversely, assume that problem~\eqref{Prox} is dual-representable for every $z \in \bX $.
Let $x \in \F $ and $u \in \partial(\fE)(x)$ be arbitrary. 
Set $z := x + \lambda u$. Then we have
\[
0 \in \partial(\fE)(x) + \tfrac{1}{\lambda}(x - z)
\quad \Longleftrightarrow \quad
x \in \Prox_{\lambda \fE}(z).
\]
Since~\eqref{Prox} is dual-representable, there exist 
$y' \in \bY$ and $x' = \Prox_{\lambda f}(z +  A^*y')$ such that $Ax' = b$.
The optimality condition of the proximal mapping implies
\[
0 \in \partial f(x') + \tfrac{1}{\lambda}\big(x' - (z +  A^*y')\big)
\quad \Longleftrightarrow \quad
0\in   \partial f(x') - \tfrac{1}{\lambda}A^*y'+\tfrac{1}{\lambda}(x'-z).
\]
Since $A^*y' \in \cN_{\E}(x')$, it follows that
\begin{equation}\label{eq:partial-inclusion}
    0\in \big( \partial f(x') + \cN_{\E}(x') \big)+\tfrac{1}{\lambda}(x'-z)\subseteq \partial(  \fE)(x') +\tfrac{1}{\lambda}(x'-z),
\end{equation}
where the last inclusion follows from~\cite[Theorem~23.8]{rockafellar1970convex}. Hence, we have $x'\in\Prox_{\lambda \fE}(z)$.
By the uniqueness of the convex proximal subproblem, we conclude that $x = x'$. Applying \eqref{eq:partial-inclusion} again yields
\[
u = \tfrac{1}{\lambda}(z - x)
\in \partial f(x) + \cN_{\E}(x).
\]
Since $u$ is arbitrary in $\partial(\fE)(x)$, we have
\[
 \partial(\fE)(x)\subseteq \partial f(x) + \cN_{\E}(x), \quad \forall\, x \in \F .
\]
Because the reverse inclusion 
always holds from~\cite[Theorem~23.8]{rockafellar1970convex}, we conclude that SSR holds for all \(x \in \F \).
\end{proof}

While SSR constitutes the weakest regularity condition ensuring the dual representability of problem \eqref{Prox}, directly verifying SSR in concrete applications can be nontrivial. Therefore, we recall several stronger but more tractable regularity assumptions that automatically imply SSR.

{
The next proposition collects several classical regularity conditions that are commonly invoked in variational analysis. Importantly, in the context of dual-representability, we show that these conditions form a strict hierarchy, culminating in SSR as the weakest sufficient requirement. 
}

\begin{proposition}[\bf Hierarchy of regularity conditions]\label{prop:regularity-hierarchy}
Under Assumption~\ref{assump:cvx}, consider the following conditions at $x\in\F $:
\begin{enumerate}
[label=\bf (\arabic*)]
\item Subdifferential sum rule (SSR):
\begin{equation*}
    \partial(\fE)(x) = \partial f(x) + A^*\bY .
\end{equation*}
\item Restricted subdifferential sum rule (R-SSR):
\begin{align*}
    \partial(\fE)(x)
        &= \partial f(x) + A^*\bY,\\
        \partial^\infty (\fE)(x)&=\partial^\infty f(x) + A^*\bY.
\end{align*}
    \item Epigraphical strong conical hull intersection property (epi-CHIP):
    \begin{equation*}
        \cN_{\epi \fE}(x,f(x))=\cN_{\epi f}(x,f(x)) + \cN_{\E\times \R}(x,f(x)).
    \end{equation*}
\item   Epigraphical transversality:
\[
\cN_{\epi f}(x,f(x)) \cap \cN_{\E\times\mathbb{R}}(x,f(x)) = \{0\},
\]
or equivalently,
\[
\T_{\epi f}(x,f(x)) + \T_{\E\times\mathbb{R}}(x,f(x)) = \bX \times \mathbb{R}.
\]
\item  Horizon qualification condition:
\[
\partial^{\infty} f(x) \cap A^{*}\bY = \{0\}.
\]
\item   Robinson constraint qualification (RCQ):
\begin{equation}\label{RCQ}
    \T_{f}(x)+\ker A=\bX
\quad \Longleftrightarrow\quad 
A\T_{f }(x) = \bY,
\end{equation}
where 
\[
\T_{f }(x):=\{v\mid(v,h)\in \T_{{\rm epi}\,f}(x,f(x))\}=\dom {\rm d}f(x),
\]
and the last equality follows from \cite[Theorem 8.2]{rockafellar1998variational}. 
\item Robinson constraint nondegeneracy:
\[
\lin \; \T_{f}(x)+\ker A=\bX
\quad \Longleftrightarrow\quad A\big(\lin\; \T_{f}(x)\big) = \bY.
\]
\end{enumerate}
The following implications hold:
\[
\begin{aligned}
(1)
\; \Longleftarrow\;
\big[(2)
\; \Longleftrightarrow\; (3)\big]
\;& {\Longleftarrow}\;  \big[ (4)
\; \Longleftrightarrow\; (5)\; {\Longleftrightarrow}\; (6)\big]
\; \Longleftarrow\; (7).
\end{aligned}
\]
\end{proposition}

\begin{proof}
The implication $(1)\Leftarrow(2)$ is immediate.  
The implication $(3)\Leftarrow(4)$ follows from the normal–cone intersection property;  
see, for example,~\cite[Theorem~6.42]{rockafellar1998variational}.  
The equivalence  $(4) \Leftrightarrow(5)$ follows from  
\cite[Theorem~8.9]{rockafellar1998variational}, which connects epigraphical normals with  
(horizon) subgradients. The equivalence \((4)\Leftrightarrow(6)\) follows from the tangent-cone description of (4) together with the identity
\(\T_{\E\times\R}(x,f(x)) = \ker A \times \R\).
The implication $(6)\Leftarrow(7)$ follows from standard  
results concerning constraint qualifications. 

We now prove the equivalence between (2) and (3). 
Recall that $\fE = f + \delta_{\E}$ and 
\[
\epi \fE = \epi f \cap \big(\E\times \R\big),\quad 
\cN_{\E\times\R}(x,f(x)) = \cN_{\E}(x)\times\{0\}.
\]

\noindent{(2) $\Rightarrow$ (3).}
Assume that R--SSR holds at $x$.
Let $(v,\alpha)\in \cN_{\epi \fE}(x,f(x))$ be arbitrary.  
By \cite[Theorem~8.9]{rockafellar1998variational}, either $\alpha<0$ and $(v,\alpha)=\lambda(\tilde v,-1)$ for some $\lambda>0$, $\tilde v\in\partial \fE(x)$, or $\alpha=0$ and $(v,0)$ is a horizon normal with $v\in\partial^\infty \fE(x)$.

Consider first the case $\alpha<0$.  
Then $(v,\alpha)=\lambda(\tilde v,-1)$ with $\tilde v\in\partial \fE(x)$.  
By R--SSR, we can write $\tilde v = v_f + v_A$ with $v_f\in\partial f(x)$ and $v_A\in A^*\bY =\cN_{\E}(x)$.  
Hence
\[
(v,\alpha)
= \lambda(v_f,-1) + \lambda(v_A,0)
\in \cN_{\epi f}(x,f(x)) + \cN_{\E}(x)\times\{0\}.
\]
The case $\alpha=0$ is analogous: if $v\in\partial^\infty \fE(x)$, then by R--SSR
$v = v_f^\infty + v_A^\infty$ with $v_f^\infty\in\partial^\infty f(x)$ and 
$v_A^\infty\in A^*\bY =\cN_{\E}(x)$. Thus
\[
(v,0) = (v_f^\infty,0) + (v_A^\infty,0)
\in \cN_{\epi f}(x,f(x)) + \cN_{\E}(x)\times\{0\}.
\]
Therefore, we obtain
\[
\cN_{\epi \fE}(x,f(x))
\subset \cN_{\epi f}(x,f(x)) + \cN_{\E}(x)\times\{0\}.
\]
The reverse inclusion always holds for convex sets (intersection of $\epi f$ and $\E\times\R$)~\cite[Theorem~23.8]{rockafellar1970convex}, so equality follows and (3) holds.

\medskip

\noindent{(3) $\Rightarrow$ (2).}
Conversely, assume that epi--CHIP holds at $x$.
Let $u\in\partial \fE(x)$ be arbitrary.  
By \cite[Theorem~8.9]{rockafellar1998variational}, $(u,-1)\in\cN_{\epi \fE}(x,f(x))$, so there exist
$(v_f,-1)\in\cN_{\epi f}(x,f(x))$ and $(v_A,0)\in\cN_{\E}(x)\times\{0\}$ such that
\[
(u,-1) = (v_f,-1) + (v_A,0).
\]
By \cite[Theorem~8.9]{rockafellar1998variational} again, $v_f\in\partial f(x)$. 
Hence $u=v_f+v_A\in \partial f(x)+A^*\bY $.  
Since $u\in\partial \fE(x)$ is arbitrary, we conclude that
\[
\partial \fE(x) \subset \partial f(x) + A^*\bY .
\]
The reverse inclusion always holds for convex functions (because $\fE=f+\delta_{\E}$)~\cite[Theorem~23.8]{rockafellar1970convex}, so we obtain
$\partial \fE(x)=\partial f(x)+A^*\bY $.

The same argument with $(u,0)$ and the horizon normals in \cite[Theorem~8.9]{rockafellar1998variational} shows that
\[
\partial^\infty \fE(x) = \partial^\infty f(x) + A^* \bY .
\]
Thus R--SSR holds at $x$, and (2) follows. 
This proves the equivalence of (2) and (3).
\end{proof}

\begin{remark}
    Robinson constraint nondegeneracy ensures the uniqueness of the multiplier $y^*$ associated with the inclusion problem~\eqref{eq:dual-repre}, and hence the uniqueness of the dual representation. 
\end{remark}

\begin{remark}
The epi-CHIP is not identical to the classical epigraphical extension of strong CHIP~\cite{deutsch1997dual}, as we have replaced $\mathbb{R}_{+}$ by $\mathbb{R}$ in the vertical direction to align the geometric condition with the subdifferential sum rule (SSR). This modification slightly relaxes the classical intersection structure; however, when $f=\delta_C$ for a closed convex set $C$, the epi-CHIP reduces exactly to the classical strong CHIP for the pair $(C,\E)$. The following example shows that the replacement of $\mathbb{R}_{+}$ by $\mathbb{R}$ is in fact necessary: if one keeps $\mathbb{R}_{+}$, the epigraphical intersection may become empty even though the SSR still holds.
\end{remark}

\begin{example}[\bf Necessity of using $\mathbb{R}$ instead of $\mathbb{R}_+$]
\label{ex:epi-CHIP-R}
Let $\bX=\mathbb{R}$, $A(x)=x$, $b=0$, and $f(x)\equiv -1$. 
Then $\operatorname{epi}f=\{(x,t)\mid t\ge -1\}$ and 
$\E\times\mathbb{R}_{+}=\{(0,t)\mid t\ge0\}$. 
At \((x,f(x))=(0,-1)\), the point does not belong to the classical epigraphical intersection \(\operatorname{epi}f\cap(\E\times\mathbb{R}_{+})\), so the normal cone of this set at \((0,-1)\) is not defined, 
although the SSR,
$\partial(\fE)(x)=\partial f(x)+A^{*}\bY$, still trivially holds. 
Replacing $\mathbb{R}_{+}$ by $\mathbb{R}$ restores a meaningful intersection
and recovers the geometric equivalence with the SSR, as illustrated in Figure~\ref{fig:epi-intersection}.

\begin{figure}[ht!]
    \centering
    \includegraphics[width=1\textwidth]{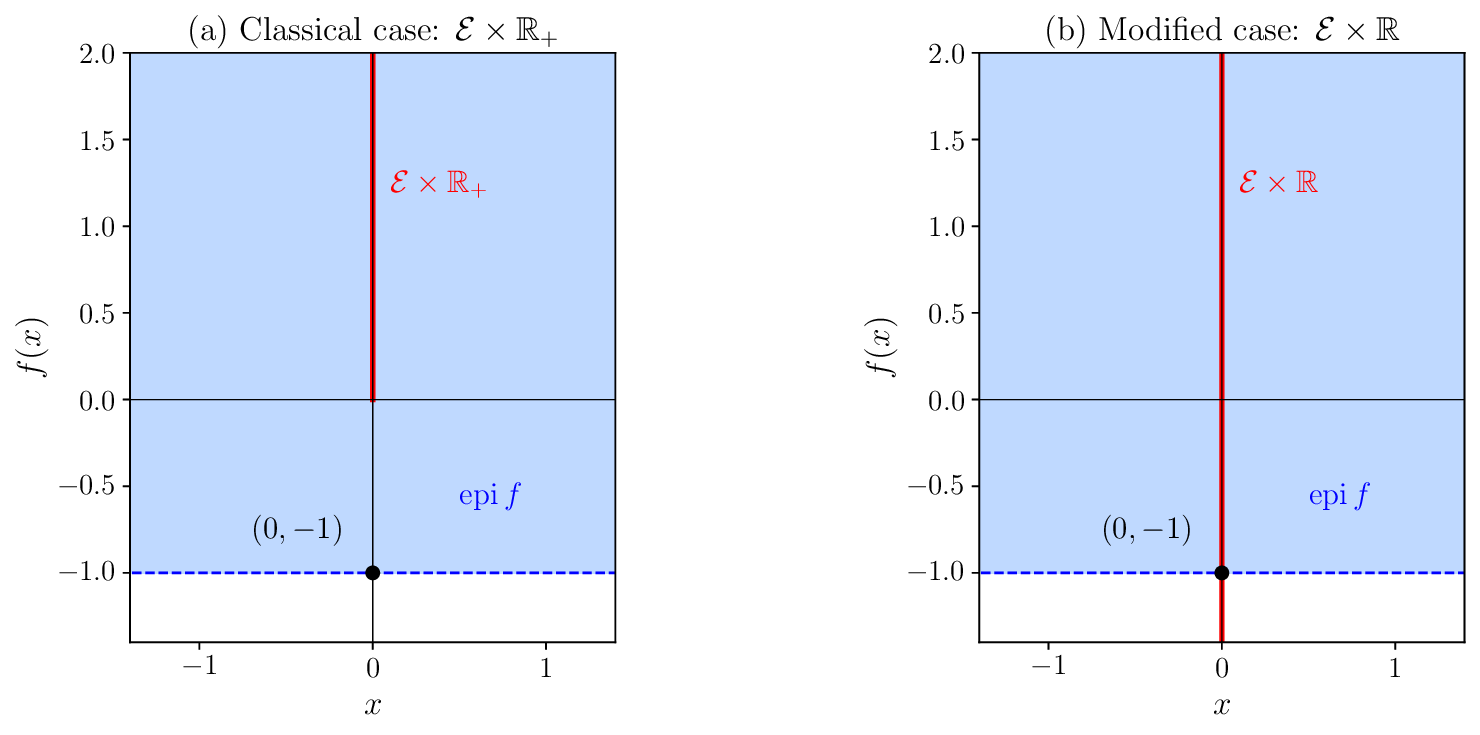}
    \caption{Epigraphical intersection in Example~\ref{ex:epi-CHIP-R}: 
replacing $\mathbb{R}_{+}$ by $\mathbb{R}$ restores a well-defined 
normal cone at $(0,-1)$
and ensures SSR.
}
\label{fig:epi-intersection}
\end{figure}

\end{example}

Based on Proposition~\ref{prop:SD-eq}, ~\ref{prop:basic-props-dual} and~\ref{prop:regularity-hierarchy}, the relationships among the various regularity conditions in the convex case are illustrated in Figure~\ref{fig:regularity-hierarchy}.

\begin{figure}[ht!]
\centering
\begin{tikzpicture}[
  node distance=7mm,
  every node/.style={font=\small, align=center},
  block/.style={
    draw,
    fill=gray!6,
    text width=11cm,
    inner sep=5pt,      
    inner ysep=6pt,     
    rounded corners=2pt
  },
  arrow/.style={{Latex}-, thick}
]

\node[block] (A) {\textbf{Dual-representability:}\\[-3pt]
    \hrulefill\\[2pt]
     Attainable strong duality\\[-2pt]
  \(\Updownarrow\) \\
     Lagrangian saddle point\\[-2pt]
  \(\Updownarrow\) \\
     Dual inclusion\\[-2pt]
  \(\Updownarrow\) \\
     Subdifferential sum rule (SSR)
};

\node[block, below=of A] (C) {
  Restricted subdifferential sum rule (R-SSR)\\[-2pt]
  \(\Updownarrow\) \\
  Epigraphical strong conical hull intersection property (epi-CHIP)
};

\node[block, below=of C] (D) {
  Epigraphical transversality \\[-2pt]
  \(\Updownarrow\) \\
  Horizon qualification condition\\[-2pt]
  \(\Updownarrow\) \\
  Robinson constraint qualification (RCQ)
};

\node[block, below=of D] (E) {
  Robinson constraint nondegeneracy
};

\draw[arrow] (A.south) -- (C.north);
\draw[arrow] (C.south) -- (D.north);
\draw[arrow] (D.south) -- (E.north);

\end{tikzpicture}
\caption{Convex hierarchy of regularity conditions.}
\label{fig:regularity-hierarchy}
\end{figure}
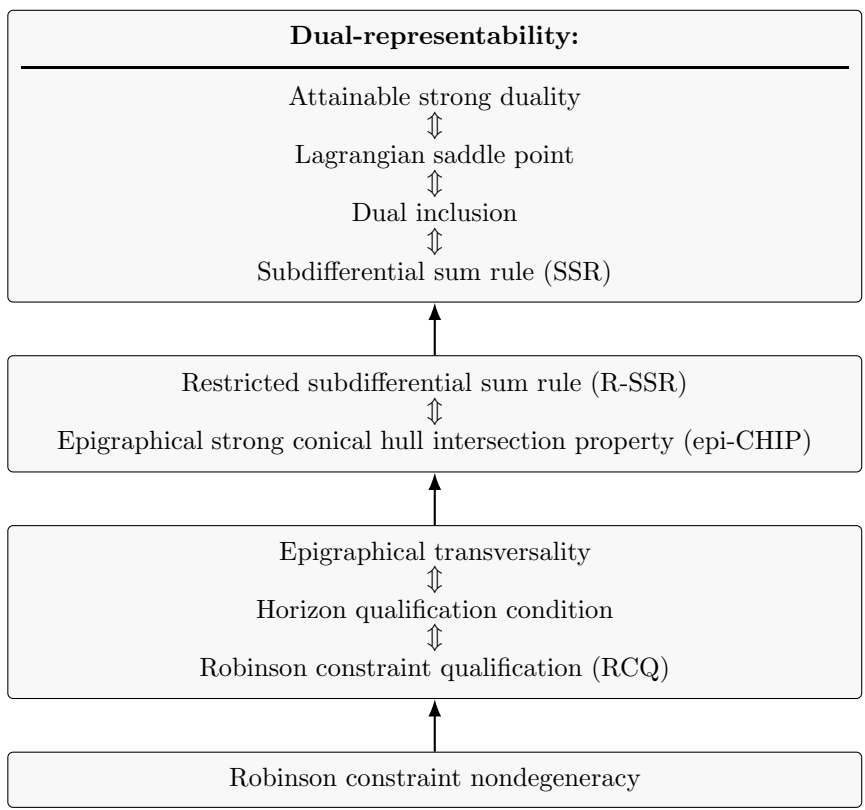

{
We note that the hierarchy of regularity conditions in Proposition~\ref{prop:regularity-hierarchy} extends similarly to nonconvex functions by replacing equalities with inclusions and using limiting subgradients or normal cones. However, without convexity, SSR no longer suffices to guarantee dual-representability, and additional structure is needed. We develop this extension in the next subsection.
}


\subsection{Nonconvex proximal mapping}\label{subsec:nonconvex-prox}

In this subsection, we study proximal mappings associated with a general, possibly nonconvex function \( f\). 
Beyond the convex setting, the dual representability of \eqref{Prox} requires additional variational properties. The key notion that enables this extension is prox-regularity, which provides a local equivalence between subdifferential inclusions and the proximal mapping.
For completeness, we recall its definition below; see \cite[Definition~13.27]{rockafellar1998variational}. Several other equivalent characterizations of prox-regular functions can be
found in \cite[Theorem~1.3]{poliquin2000local}.

\begin{definition}[\bf Prox-regularity]\label{def:prox-regular}
A function $f:\bX\rightarrow\overline{\mathbb{R}}$ is prox-regular at $\bar{x}$ for $\bar{v}$ if $f$ is finite and locally lsc at $\bar{x}$ with $\bar{v}\in\partial f(\bar{x})$, and there exist $\epsilon>0$ and $\rho\geq 0$ such that
$$f(x')\geq f(x)+\langle v,x'-x\rangle-\frac{\rho}{2}\|x'-x\|^2\quad\mbox{for all}\;\;x'\in\mathbb{B}_{\epsilon}(\bar{x})$$
when 
$x\in\mathbb{B}_{\epsilon}(\bar{x})\cap\{ x\mid f(x)<f(\bar{x})+\epsilon\}$, $v\in\partial f(x) \cap \mathbb{B}_{\epsilon}(\bar{v})$. A set ${\cal S}$ is prox-regular {at $\bar x$ for $\bar v$ if the indicator function}  $\delta_{\cal S}$ is prox-regular {at $\bar x$ for $\bar v$}. 
\end{definition}

Prox-regularity is a widely used variational property that relaxes weak convexity. Many commonly encountered functions are prox-regular, including all (weakly) convex functions \cite[Example 13.30]{rockafellar1998variational}, indicator functions $\delta_{\mathcal K}(g(x))$ under RCQ \cite[Proposition 2.2]{levy2000stability}, lower $C^{2}$ and more generally $C^{1,1}$ functions \cite[Proposition 13.34]{rockafellar1998variational}, as well as strongly amenable functions \cite[Example 13.32]{rockafellar1998variational}.

Next, we introduce the definition of nonconvex subdifferential sum rule (NSSR) of functions, which is an extension of SSR for nonconvex functions. 
\begin{definition}[\bf Nonconvex subdifferential sum rule]
Let $f, g : \bX \rightarrow \Rbar $ be lsc
functions. 
The nonconvex subdifferential sum rule (NSSR)
is said to hold for $f$ and $g$ at $\bar x\in \dom f \cap \dom g$ if there exist $\tau,\delta\in (0,+\infty)$ such that 
\begin{equation}\label{eq:ext}
            \partial ( f+g)({x})
            \cap\mathbb{B}_d (0)\subseteq\partial  f({x})\cap\tau\mathbb{B}_d (0)+\partial  g(x)\cap\tau\mathbb{B}_d (0)\end{equation}
        holds for all $x\in\mathbb{B}_\delta(\bar{x})\cap \dom f\cap \dom g$ and $d\in(0,1]$.
\end{definition}

\begin{remark}
    If \(f\) and \(g\) are indicator functions of two closed sets, then NSSR is implied by the subtransversality of the underlying sets; see \cite[Theorem~3.3]{wei2024subtransversality}.
\end{remark}

We focus on the pair $(f,\delta_\E)$. 
Similar to the convex case, NSSR is a relatively weak condition and can be implied by several other constraint qualifications. We characterize some of these conditions below.

\begin{proposition}[\bf Hierarchy of regularity conditions]\label{Hierarchy-nonconvex}
   Under Assumption~\ref{assump:basic},
consider the following conditions at $\bar x\in\F$:
    \begin{enumerate}
    [label=\bf (\arabic*)]
    \item Nonconvex subdifferential sum rule (NSSR) of $f$ and $\delta_\E$ at $\bar x$.  
    \item Epigraphical subtransversality: there exist $\tau,\delta\in (0,+\infty)$ such that
        \begin{equation*}
            \dist (t,\epi \fE)
        \;\leq\; \tau 
        \dist(t,\epi f)
        +\tau \dist(t,\E\times \mathbb{R})
        \qquad \forall\; t\in \bB_{\delta}(\bar x,f(\bar x)).
        \end{equation*}
        \item Epigraphical transversality: 
        \[
        \cN_{\epi f}(\bar x, f(\bar x)) \cap \cN_{\E\times\mathbb{R}}(\bar x, f(\bar x)) = \{0\},
        \]
        \item Horizon qualification condition:
        \[
\partial^{\infty} f(\bar x) \cap A^{*}\bY = \{0\}.
\]
\item   Robinson constraint qualification (RCQ):
\begin{equation}\label{eq:epircq-nonconconvex}
    \T_{f}(\bar x)+\ker A=\bX
\quad \Longleftrightarrow\quad 
A\T_{f }(\bar x) = \bY,
\end{equation}
{where 
\[
\T_{f }(\bar x):=\{v\mid(v,h)\in \T_{{\rm epi}\,f}(\bar x,f(\bar x))\}=\dom {\rm d}f(\bar x),
\]}
and the last equality follows from \cite[Theorem 8.2]{rockafellar1998variational}. 
\item Robinson constraint nondegeneracy:
\[
\lin \; \T_{f}(\bar x)+\ker A=\bX
\quad \Longleftrightarrow\quad A\big(\lin\; \T_{f}(\bar x)\big) = \bY.
\]
    \end{enumerate}
The following implications hold:
\[
\begin{aligned}
(1)
\; \Longleftarrow\;
(2)
\; \Longleftarrow\; 
\big[(3)
\;\Longleftrightarrow\;  
(4)
\; \overset{\text{regular}}{\xLongleftrightarrow{\quad}}
\; 
(5)\big]
\Longleftarrow\; 
(6),
\end{aligned}
\]   
where ``regular" denotes that $f$  is epigraphically regular at $\bar x$.
\end{proposition}

\begin{proof}
The implication \((2)\Leftarrow(3)\) is well-known in 
the transversality literature, see, for example,~\cite{kruger2017subtransversality}.
The equivalence \((3)\Leftrightarrow(4)\) follows from \cite[Theorem~8.9]{rockafellar1998variational}, which links epigraphical normals with (horizon) subgradients.  
Under epigraphical regularity, the equivalence \((3)\Leftrightarrow(5)\) follows by taking the polar of \eqref{eq:epircq-nonconconvex}.  
The implication \((5)\Leftarrow(6)\) is a direct consequence of standard constraint qualification results.

We now establish the implication \((1)\Leftarrow(2)\). Assume that subtransversality holds for ${\rm epi} f$ and $\E\times\R$ at $\bar x$. We only treat the case \(d=1\) since the case \(d\in(0,1)\) follows immediately by a scaling argument. It follows from \cite[Theorem~3.3]{wei2024subtransversality} that there exists $\tau, \delta\in(0,+\infty)$ such that 
\begin{equation}\label{eq:wst}
{\cal N}_{\epi \fE}(x, \eta)\cap\mathbb{B}_1(0)\subseteq\tau({\cal N}_{{\rm epi}\,f}(x,\eta)\cap\mathbb{B}_1(0)+{\cal N}_{\E\times\mathbb{R}}(x,\eta)\cap\mathbb{B}_1(0))
\end{equation}
holds for all {$(x,\eta)\in\mathbb{B}_{\delta}(\bar{x},f(\bar x))\cap \epi \fE$. }
Suppose \(v\in\partial \fE({x})\). By \cite[Theorem~8.9]{rockafellar1998variational}, we have
$$
\left(\frac{v}{\sqrt{\|v\|^2+1}},\frac{-1}{\sqrt{\|v\|^2+1}}\right)\in {\cal N}_{\epi \fE}({x}, f( x))\cap\mathbb{B}_1(0)
$$
Combined with \eqref{eq:wst}, 
there exists $(v_1,w_1)\in {\cal N}_{{\rm epi}\,f}({x},f( x))\cap\tau\mathbb{B}_1(0) $ and $
  (v_2,w_2)\in \big({\cal N}_{\E}({x})\times\{0\}\big)\cap\tau\mathbb{B}_1(0),$ such that
\[
\left(\frac{v}{\sqrt{\|v\|^2+1}},\frac{-1}{\sqrt{\|v\|^2+1}}\right)=(v_1,w_1)+(v_2,w_2).
\]
Since \(w_2=0\), we obtain \(w_1=-1/\sqrt{\|v\|^{2}+1}\).  
By \cite[Theorem~8.9]{rockafellar1998variational},
we obtain the scaled inclusions
\[
\sqrt{\|v\|^{2}+1}\,v_1\in\partial f( x)\cap\bigl(\tau\sqrt{\|v\|^{2}+1}\,\mathbb{B}_1(0)\bigr),
\quad
\sqrt{\|v\|^{2}+1}\,v_2\in{\cal N}_{\E}( x)\cap\bigl(\tau\sqrt{\|v\|^{2}+1}\,\mathbb{B}_1(0)\bigr).
\]
Hence
\[
v\in\partial f( x)\cap\bigl(\tau\sqrt{\|v\|^{2}+1}\,\mathbb{B}_1(0)\bigr)
+{\cal N}_{\E}( x)\cap\bigl(\tau\sqrt{\|v\|^{2}+1}\,\mathbb{B}_1(0)\bigr),
\]
and therefore
\[
\partial(\fE)( x)\cap\mathbb{B}_1(0)
\subseteq
\partial  f( x)\cap \sqrt{2}\tau\mathbb{B}_1(0)
+{\cal N}_{\E}( x)\cap \sqrt{2}\tau\mathbb{B}_1(0).
\]
Thus, up to a rescaling of constants, the subtransversality of $f$ and $\E$ can be derived from the subtransversality of ${\rm epi}\,f$ and $\E\times\mathbb{R}$. 
\end{proof}

\begin{remark}
    RCQ ensures the nonemptiness and compactness of the multiplier set, and the stronger nondegeneracy condition further guarantees the uniqueness of the multiplier.
\end{remark}

We are now ready to state the main result of this subsection. Building upon the prox-regularity framework and the nonconvex subdifferential sum rule developed above, the following theorem establishes the local dual representability of the nonconvex proximal mapping.

\begin{theorem}[\bf Nonconvex dual representability]\label{Thm:nonconvex-dual-repre}
    Suppose Assumption~\ref{assump:basic} holds, and that given $\bar{x}\in{\cal F}$:
\begin{enumerate}[label=\bf {(\roman*)}]
\item $ f$ and $\delta_\E$ satisfies NSSR at $\bar x$.
\item $f$ is prox-regular at $\bar{x}$ for $\bar v=0$ with $\rho\in (0,1/\lambda)$ and $\bar x =\Prox_{f/\rho}(\bar x)$. 
\end{enumerate}
Then \eqref{Prox} is dual-representable near $z=\bar x$.
\end{theorem}

\begin{proof}
    Assume that $f$ is prox-regular at $\bar{x}$ for $\bar v=0$ with respect to $\epsilon_1>0$ and $\rho<1/\lambda$ from Definition~\ref{def:prox-regular}. 
    It follows from~\cite[Theorem~4.4]{poliquin1996prox} that there exists a constant $\epsilon_2>0$ such that, for all $x\in\bB_{\epsilon_2}(\bar x)$, $y\in\bB_{\epsilon_1}(\bar x)$ and $v\in\partial f(y)\cap \bB_{\epsilon_1}(0)$ satisfying 
    $|f(y)-f(\bar x)|\le \epsilon_1$,
the following implication holds:
\begin{equation}\label{prox-implication}
y+\lambda v= x
\quad\Longrightarrow\quad
y=\Prox_{\lambda f}(x).
\end{equation}
    Assume that the NSSR of $f$ and $\delta_\E$ holds at $\bar{x}\in \dom \fE$ with the constants $\delta>0$ and $\tau>0$ from~\eqref{eq:ext}. 
    Since $\rho<1/\lambda$ and $\bar x=\Prox_{f/\rho}(\bar x)$, we have $\bar x=\Prox_{\lambda f}(\bar x)$. Moreover, $\bar x\in \dom \fE$ implies that $\bar x=\Prox_{\lambda f+\delta_\E}(\bar x)=\Prox_{\lambda \fE}(\bar x)$, and hence $0\in \partial(\fE)(\bar x)$. By the definition of prox-regularity and NSSR assumption, $ f+\delta_\E$ is also prox-regular at $\bar{x}$ for $\bar v=0 $ with respect to $\epsilon_3\coloneqq  \min\{\epsilon_1,\epsilon_1/\tau\}$ and $\rho<1/\lambda$.
    It follows from~\cite[Theorem 4.4]{poliquin1996prox} that there exists $\epsilon_4>0$ such that $\Prox_{\lambda \fE }$ is single-valued and Lipschitz continuous with constant $L\coloneqq1/(1- \rho\lambda)> 1$ in $\bB_{\epsilon_4}(\bar x)$. Set
    \begin{equation}\label{eps-z}
        \epsilon_z=\min\left\{
        \frac{\lambda}{L+1},
        \frac{\delta}{L+1},
        \frac{\epsilon_1}{L},
        \frac{\sqrt{\epsilon_1}}{L+2},
        \frac{\lambda \epsilon_1}{\tau (L+1)},
        \frac{\epsilon_2}{1+\tau (L+1)},
        \epsilon_4
\right\}.
    \end{equation}
Let $z\in\mathbb{B}_{\epsilon_z}(\bar{x})$ be arbitrary and $\hat{x}$ being the optimal solution of \eqref{Prox} at $z$. Next we show that \eqref{Prox} is dual-representable at $z$.
The first-order necessary condition~\cite[Theorem~10.1]{rockafellar1998variational} for \eqref{Prox} yields
\begin{equation}\label{eq:fonc}
-(\hat{x}-z)\in\lambda\partial  (\fE )(\hat{x}). \end{equation}
By the local Lipschitz continuity of $\Prox_{\lambda \fE}$, we have
\begin{align*}
    \|\hat x-\bar x\|
    &=
    \|\Prox_{\lambda \fE}(z)-\Prox_{\lambda \fE}(\bar x)\|
        \leq 
    L \|z-\bar x\|\leq L\epsilon_z,
    \\
    \|\hat x -z\|
    &\leq 
    \|\hat x-\bar x\| + \|z-\bar x \|
    \leq (L+1)\epsilon_z\leq \delta.
\end{align*}
By \eqref{eq:fonc} and the property~\eqref{eq:ext} of NSSR, for $d=({L+1})\epsilon_z/{\lambda}\in(0,1]$, we have
\begin{equation*}
-(\hat{x}-z)
\in \partial ( \lambda \fE)(\hat x)\cap \lambda \bB_{d}(0)
\subseteq
\partial  (\lambda f)(\hat{x})\cap(\tau\lambda)\mathbb{B}_d (0)+A^*\bY\cap(\tau\lambda)\mathbb{B}_d (0).
\end{equation*}
Thus, there exists $\hat{y}\in \bY$ and $v\in \partial  f(\hat{x}) $ such that 
\begin{equation}\label{eq:0y}
\|A^*\hat y\|\leq \tau \lambda d,\quad
\|v\|\leq  {\tau d},\quad 
\hat x +\lambda v= z+A^*\hat y.
\end{equation}
We now verify that all the conditions required to apply \eqref{prox-implication} are satisfied. Since
\begin{align*}
    \hat x =\Prox_{\lambda \fE}(z)
    &\quad \Longrightarrow\quad 
    \lambda f(\bar x)+\frac{1}{2}\|\bar x-z\|^2\geq \lambda f(\hat x)+\frac{1}{2}\|\hat x-z\|^2,\\
    \bar x=\Prox_{\lambda \fE}(\bar x) 
    &\quad \Longrightarrow\quad
    \lambda f(\hat x)+\frac{1}{2}\|\hat x-\bar x\|^2\geq \lambda f(\bar x)+\frac{1}{2}\|\bar x-\bar x\|^2=\lambda f(\bar x),
\end{align*}
we have 
\[
|f(\hat{x})-f(\bar x)|
\le 
\frac{1}{2\lambda}
(\|\bar x-z\|+\|\hat x-z\|)^2
\le (\epsilon_z+(L+1)\epsilon_z)^2
\le 
\epsilon_1.
\]
Moreover, $\|v\|\le \tau d=(L+1)\tau \epsilon_z/\lambda\le \epsilon_1$, $\|\hat{x}-\bar x\|
    \le
    L\epsilon_z\leq \epsilon_1$ and
\begin{align*}
    &\|(z+A^* \hat{ y})-\bar x\|
    \le 
    \|z-\bar x\|+\|A^*\hat y\|
    \le 
    {\epsilon_z}+\tau \lambda d
    =
    (1+\tau(L+1))\epsilon_z
    \le 
    \epsilon_2,
\end{align*}
we can now apply \eqref{prox-implication}, which indicates that
$$
\hat x +\lambda v= z+A^*\hat y
\quad\Longrightarrow\quad 
\hat{x}={\rm Prox}_{\lambda f}(z+A^* \hat{ y})
\quad\Longrightarrow\quad 
A {\rm Prox}_{\lambda f}(z+A^* \hat{ y})=b. 
$$
Thus, we have verified that \eqref{Prox} is dual-representable at $z$. Since $z$ is arbitrary in the neighborhood $\bB_{\epsilon_z}(\bar x)$, it follows that \eqref{Prox} is dual-representable near $\bar x$.
\end{proof}

\begin{remark}[\bf Reduction to the zero-subgradient case]\label{remark-v}
The standing assumption $\bar v = 0 \in \partial f(\bar x)$ is imposed only for notational convenience and entails no loss of generality. Indeed, for a general problem~\eqref{Prox}, let $\bar v \in \partial f(\bar x)$ be arbitrary, and define the shifted function
\[
g(x)\coloneqq f(x)-\langle \bar v, x-z\rangle .
\]
Let $g_{\E} = g+\delta_\E$.
Then the proximal problem~\eqref{Prox} associated with $f$ can be equivalently rewritten as 
\[
E_{\lambda \fE}(z)
=
E_{\lambda g_{\E}}(z-\lambda \bar v)
=\min_{x\in \bX}\left\{ \lambda g(x)+\frac{1}{2}\|x-(z-\lambda \bar v)\|^2 \;\middle|\; Ax=b\right\}.
\]
Moreover, by construction, we have $0\in \partial g(\bar x)$, so  the transformed problem satisfies the standing assumption $\bar v_g=0$. Consequently, all results established under the condition $\bar v=0$ apply directly to the original problem.
\end{remark}

\begin{remark}[\bf Practical scope of the assumptions]
Building on Remark~\ref{remark-v}, we now comment on the practical scope of the remaining assumptions in Theorem~\ref{Thm:nonconvex-dual-repre}. These assumptions are generally mild and nonrestrictive in applications. In particular, the NSSR condition is a weak regularity requirement and is implied by many standard constraint qualifications, as summarized in Proposition~\ref{Hierarchy-nonconvex}.  
Similarly, the prox-regularity assumption is standard in variational analysis and is consistent with the standing assumptions adopted throughout the prox-regularity literature; see, for instance, \cite{poliquin2010calculus, poliquin1996prox, rockafellar1998variational}. Consequently, both NSSR and prox-regularity are satisfied in a broad class of problems of practical interest.  

From an algorithmic perspective, for projected gradient-type methods with sufficiently small stepsizes $\lambda_k$, the iterates typically remain in a neighborhood of a reference point $\bar x$ where the required local assumptions hold. Hence, the conditions ensuring local dual representability are naturally fulfilled along the iteration.  
Moreover, when the proximal operator reduces to the metric projection, several of the assumptions simplify automatically, and it suffices to require only the prox-regularity of the underlying set. In this setting, Theorem~\ref{Thm:nonconvex-dual-repre} further reduces to the simplified form given in Theorem~\ref{Thm:nonconvex-dual-repre-proj}. 
\end{remark}

\begin{remark}[\bf Beyond local dual representability]
    In practice, the region of $z$ for which dual representability holds may be substantially larger than the local neighborhood, and can even coincide with the entire space, as illustrated by the sparse simplex projection problem in Subsection~\ref{example:sparse-simplex}.
\end{remark}

As a direct consequence of Theorem~\ref{Thm:nonconvex-dual-repre}, we obtain the following corollary, which requires substantially weaker assumptions when the parameter $\lambda$ is allowed to vary.

\begin{corollary}[\bf Dual representability with $\lambda$]\label{corollary:nonconvex-dual-repre-lambda}
    Suppose Assumption~\ref{assump:basic} holds, and that given $\bar{x}\in{\cal F}$:
\begin{enumerate}[label={\bf (\roman*)}]
\item $ f$ and $\delta_\E$ satisfies NSSR at $\bar x$.
\item $f$ is prox-regular at $\bar{x}$ for $\bar v=0$.
\end{enumerate}
Then for all $\lambda>0$ sufficiently large, \eqref{Prox} is dual-representable near $\bar x$.
    
\end{corollary}

\begin{proof}
   The conclusion follows immediately from \cite[Proposition~13.37]{rockafellar1998variational} together with Theorem~\ref{Thm:nonconvex-dual-repre}.
\end{proof}


\subsection{Strong convexity of the dual objective function}\label{sec:scvx}

In this subsection, we establish the strong convexity of the dual objective function $\dd$ defined in \eqref{eq:dual-function-prox} for the dual problem~\eqref{D}.
This property ensures local uniqueness of the dual solution and enables fast local convergence of first- and second-order methods for solving the unconstrained dual problem~\eqref{D}.
We first introduce the following set-valued mapping:
\begin{equation}\label{defTz}
T_z(u) := \Prox_{\lambda f}(z+u), \quad u \in \mathcal U \coloneqq A^*\mathbb{Y}.
\end{equation}

\begin{definition}[\bf strong monotonicity]
    The mapping \(T_z\) is said to be locally strongly monotone at \(\bar u \in \mathcal U\) if there exist constants \(\epsilon_{\bar u}>0\) and \(c>0\) such that
\[
\langle T_z(u_1)-T_z(u_2),\, u_1-u_2\rangle
\;\ge\;
c\,\|u_1-u_2\|^2,
\quad
\forall\, u_1,u_2 \in \mathbb B_{\epsilon_{\bar u}}(\bar u)\cap\mathcal U.
\]
It is said to be globally strongly monotone if the above inequality holds for all \(u_1,u_2 \in \mathcal U\).
\end{definition}

We next give an equivalent characterization of the strong convexity of the dual objective in terms of the strong monotonicity of the associated proximal mapping $T_z$, which is often easier to verify in practice. Accordingly, we analyze the convex and nonconvex proximal mapping cases separately in the following subsections: the convex case yields global results, whereas the nonconvex case leads to local results.


\subsubsection{Convex proximal mapping}

We first consider the convex setting, where the proximal mapping $T_z$ is single-valued and leads to a global characterization. 

\begin{theorem}
\label{thm:equivalent-strong-cvx-prox-cvx}
Under Assumption~\ref{assump:cvx}, for given $z\in\bX$, the following statements are equivalent:
\begin{enumerate}
[label=\bf (\roman*)]
\item The mapping \(T_z\) is globally strongly monotone.
\item The dual problem~\eqref{D} is globally strongly convex.
\end{enumerate}
\end{theorem}

\begin{proof}
Under Assumption~\ref{assump:cvx}, the mapping $T_z(\cdot)$ defined in \eqref{defTz} is single valued for all $u\in\cU$. Thus, it follows from Proposition~\ref{prop:basic-props-dual} that the dual objective function $\dd$ is differentiable with
\[
\nabla \dd(y) = A\,T_z(A^*y)-b.
\]
Since $A$ has full row rank, there exist constants $\kappa_1,\kappa_2>0$ such that
\begin{equation}\label{eq:kappaA}
\kappa_1\|y\| 
\le 
\|A^*y\|
\le
\kappa_2\|y\|,\qquad\forall\; y\in\bY
.
\end{equation}

\smallskip
\noindent{(i)$\Rightarrow$(ii).}
Assume that $T_z:\mathcal U\to\mathbb X$ is globally strongly monotone.
Then 
\[
\langle T_z(u_1)-T_z(u_2),\, u_1-u_2\rangle
\;\ge\;
c\|u_1-u_2\|^2,
\qquad
\forall\, u_1,u_2\in \mathcal U .
\]
Let $u_i:=A^*y_i$ for $i=1,2$.
Using $\nabla\dd(y)=AT_z(A^*y)$ and \eqref{eq:kappaA}, we obtain
\[
\begin{aligned}
\langle \nabla\dd(y_1)-\nabla\dd(y_2),\, y_1-y_2\rangle
&= \langle T_z(A^*y_1)-T_z(A^*y_2),\, A^*(y_1-y_2)\rangle \\
&\ge c\|A^*(y_1-y_2)\|^2
\;\ge\; c\kappa_1^2\|y_1-y_2\|^2 .
\end{aligned}
\]
Since $y_1,y_2\in\bY$  are arbitrary, it follows that $\dd$ is globally strongly convex, and therefore
problem~\eqref{D} is globally strongly convex.

\smallskip
\noindent{(ii)$\Rightarrow$(i).}
Conversely, suppose that problem~\eqref{D} is globally strongly convex.
Then there exists $\mu> 0$ such that
\[
\langle \nabla\dd(y_1)-\nabla\dd(y_2),\, y_1-y_2\rangle
\;\ge\;
\mu\|y_1-y_2\|^2,
\qquad
\forall\, y_1,y_2\in \mathbb{Y}.
\]
Using again $\nabla\dd(y)=AT_z(A^*y)$ and \eqref{eq:kappaA}, we obtain
\[
\langle T_z(A^*y_1)-T_z(A^*y_2),\, A^*(y_1-y_2)\rangle
\;\ge\;
(\mu/\kappa_2^2)\|A^*(y_1-y_2)\|^2.
\]
Since $A^*y_1,A^*y_2\in\mathcal U$ are arbitrary, the above inequality exactly characterizes the global strong monotonicity of $T_z$ on $\mathcal U$.
\end{proof}


\subsubsection{Nonconvex proximal mapping}

We now turn to the nonconvex setting. In contrast to the convex case, strong monotonicity of the proximal mapping and strong convexity of the dual objective can only be expected locally, under suitable regularity conditions.

\begin{theorem}
\label{thm:equivalent-strong-cvx-prox}
Under the conditions of Theorem~\ref{Thm:nonconvex-dual-repre}, for any \(z\) sufficiently close to \(\bar x\in{\cal F}\), the following statements are equivalent:
\begin{enumerate}
[label=\bf (\roman*)]
\item The mapping \(T_z\) is locally strongly monotone at \(A^*y_z\).
\item Problem~\eqref{D} is locally strongly convex at its unique optimal solution $y_z$.
\end{enumerate}
\end{theorem}

\begin{proof}
By the proof of Theorem~\ref{Thm:nonconvex-dual-repre},
under NSSR and prox-regularity,
there exist constants $\epsilon_y,\epsilon_z>0$ such that, for any
$z\in \mathbb{B}_{\epsilon_z}(\bar x)$, the dual problem~\eqref{D} admits an optimal
solution $y_z\in \mathbb{B}_{\epsilon_y}(0)$, the dual objective function $\dd$ is differentiable on
$\mathbb{B}_{2\epsilon_y}(0)$ and satisfies
\begin{equation}\label{grad-Phi-local}
    \nabla \dd(y) = A\,T_z(A^*y)-b, \qquad \forall\, y\in \mathbb{B}_{2\epsilon_y}(0),
\end{equation}
where $T_z(\cdot)$ defined in \eqref{defTz} is single valued for all $A^*y$ with $y\in\mathbb{B}_{2\epsilon_y}(0)$.
The proof is identical to that of Theorem~\ref{thm:equivalent-strong-cvx-prox-cvx}, except that all arguments are carried out locally around \(y_z\). Hence, the details are omitted.
\end{proof}


\section{Applications}\label{sec:applications}
In this section, we demonstrate how the general dual representation, strong duality, and strong convexity results established in the last two sections can be applied to specific problems.


\subsection{Projection problem}
A representative instance of \eqref{Prox} is the nonconvex projection problem introduced in \eqref{Proj}, which can be equivalently written as
\begin{equation}\label{eq:proj}
\min_{x\in\bX }\ \Big\{\frac12\|x-z\|^2 \ \Big|\ Ax=b, \; x\in \K\Big\},\tag{Proj}
\end{equation}
where $\mathcal K\subseteq \bX$ is a closed (possibly nonconvex) set, $A:\bX\to\bY$ is a surjective linear operator, and $z\in\bX$ and $b\in\bY$ are given. We assume throughout that the feasible set $\mathcal F=\mathcal K\cap \mathcal E$ is nonempty.
The associated dual problem is given by
\begin{equation}\label{eq:dual-proj}
\min_{y\in\bY }\ \dd(y)\coloneqq -\langle b,y\rangle-\tfrac12\|z\|^2+\tfrac12\|z+A^\ast y\|^2-\tfrac12\,\mathrm{dist}(z+A^\ast y,\mathcal K)^2.
\end{equation}
The subdifferential of $\dd$ admits the characterization
\[
\partial \dd(y)=\mathrm{co}\big(A\,\Pi_{\mathcal K}(z+A^\ast y)-b\big).
\]
If $\Pi_{\mathcal K}$ is single-valued at $z+A^\ast y$, then $\dd$ is differentiable at $y$ with gradient
\[
\nabla \dd(y)=A\,\Pi_{\mathcal K}(z+A^\ast y)-b.
\]
By Proposition~\ref{prop:SD-eq}, if there exists $y\in\bY$ such that 
$0\in A\,\Pi_{\mathcal K}(z+A^\ast y)- b$, 
then the projection problem \eqref{eq:proj} is dual-representable.
When $\mathcal K$ is convex, \eqref{eq:proj} is dual-representable if and only if the SSR between $\delta_\mathcal K$ and $\delta_\mathcal E$ holds, which is also equivalent to the strong CHIP condition \cite{deutsch1997dual} between $\K$ and $\E$. This regularity condition is mild and can be guaranteed by several stronger conditions listed in Proposition~\ref{prop:regularity-hierarchy}. In addition, for the projection problem \eqref{eq:proj}, since $\partial \delta_C=\partial^\infty \delta_C=\cN_C$ for any closed convex set $C$, the hierarchy of regularity conditions further reduces and yields the following stronger implication result:
\[
\begin{aligned}
\big[(1)
\; \Longleftrightarrow\;
(2)
\; \Longleftrightarrow\; (3)\big]
\;& {\Longleftarrow}\;  \big[ (4)
\; \Longleftrightarrow\; (5){\Longleftrightarrow}\; (6)\big]
\; \Longleftarrow\; (7).
\end{aligned}
\]
When $\mathcal K$ is nonconvex, under the NSSR and prox-regularity assumptions, the general dual representability result of Theorem~\ref{Thm:nonconvex-dual-repre} admits a substantially simplified form for the projection problem \eqref{eq:proj}, as stated below.

\begin{theorem}[\bf Nonconvex dual representability]\label{Thm:nonconvex-dual-repre-proj}
Given $\bar x\in \F$, assume that:
\begin{enumerate}[label=\bf {(\roman*)}]
\item $ \delta_\K$ and $\delta_\E$ satisfy NSSR at $\bar x$.
\item $\K$ is prox-regular at $\bar{x}$.
\end{enumerate}
Then \eqref{eq:proj} is dual-representable near $z=\bar x$.
\end{theorem}

The proof of Theorem~\ref{Thm:nonconvex-dual-repre-proj} follows directly from Theorem~\ref{Thm:nonconvex-dual-repre}. Since the projection problem \eqref{eq:proj} is invariant under positive scaling of the objective function, and $\mathcal K$ is prox-regular at $\bar x$ \cite[Proposition~1.2]{poliquin2000local}, the general conditions in Theorem~\ref{Thm:nonconvex-dual-repre} reduce to the simplified assumptions stated here. Moreover, similar to Proposition~\ref{Hierarchy-nonconvex}, the NSSR assumption required here is in fact mild: the subtransversality of $\mathcal K$ and $\mathcal E$ already implies the NSSR of $\delta_{\mathcal K}$ and $\delta_{\mathcal E}$; see \cite[Theorem~3.3]{wei2024subtransversality}.


\subsubsection{$C^2$ submanifold projection problem}
\label{subsec-C2-mani}

We now focus on the case where $\mathcal K$ is a $C^2$ embedded submanifold of $\mathbb R^n$.  
This class is broad and encompasses many sets of practical interest, including affine subspaces, spheres, Stiefel manifolds, and manifolds defined by smooth equality constraints.  
Under this geometric structure, $\mathcal K$ is prox-regular at every feasible point, which guarantees local dual representability under the NSSR assumption.  

Moreover, under an additional nondegeneracy condition, this setting ensures local strong convexity of the dual function~$\dd$, yielding sharper stability and convergence properties and recovering the spherical manifold case in~\cite[Proposition~1]{RiNNAL}.

\begin{theorem}[\bf Strong convexity for $C^2$ submanifold]\label{thm:strong-cvx-c2-manifold}
Assume that
\begin{enumerate}
[label=\bf (\roman*)]
    \item $\mathcal K$ is a $C^2$ embedded submanifold of $\R^n$.
    \item $\bar x\in\mathcal F=\K\cap \E$ satisfies the nondegeneracy condition.
\end{enumerate}
For $z$ near $\bar{x}$,
the dual problem~\eqref{D} is locally strongly convex at its optimal solution.
\end{theorem}

\begin{proof}
Since \(\mathcal K\) is a \(C^2\) embedded submanifold of \(\mathbb{R}^n\) and \(\bar x\in\mathcal K\), it follows from~\cite[Theorem~2]{leobacher2021existence} that the projection mapping \(\Pi_{\mathcal K}\) is \(C^1\) in a neighborhood of \(\bar x\), and from~\cite[Theorem~C]{leobacher2021existence} that the derivative of $\Pi_\K$ at \(\bar x\) is given by
$
D\Pi_{\mathcal K}(\bar x)=\Pi_{\T_{\mathcal K}(\bar x)}.
$
For \(z=\bar x\), the dual optimal solution of~\eqref{Proj} is \(y^*=0\), and hence
\[
\nabla^2\Phi_{\bar x}(0)
= A\,D\Pi_{\mathcal K}(\bar x)\,A^*
= A\,\Pi_{\T_{\mathcal K}(\bar x)}\,A^* .
\]
For any \(h\in\mathbb{Y}\),
\[
\langle h,\nabla^2\Phi_{\bar x}(0)h\rangle
= \big\|\Pi_{\T_{\mathcal K}(\bar x)}(A^*h)\big\|^2 .
\]
If \(\langle h,\nabla^2\Phi_{\bar x}(0)h\rangle=0\), then \(A^*h\in \N_{\mathcal K}(\bar x)\).
By the nondegeneracy condition,
\(\operatorname{Range}(A^*)\cap \N_{\mathcal K}(\bar x)=\{0\}\),
which implies \(A^*h=0\) and, since \(A\) has full row rank, \(h=0\).
Therefore, \(\nabla^2\Phi_{\bar x}(0)\succ0\).
Moreover, by nondegeneracy, the KKT system associated with the Lagrangian of the projection problem satisfies the assumptions of the implicit function theorem, and hence the dual solution mapping \(z \mapsto y_z\) is continuous near \(\bar x\), with \(y_z \to 0\) as \(z \to \bar x\).
Finally, since \(\Pi_{\mathcal K}\) is \(C^1\) near \(\bar x\), the mapping
\[
(z,y)\mapsto \nabla^2\Phi_z(y)
= A\,D\Pi_{\mathcal K}(z+A^*y)\,A^*
\]
is continuous in a neighborhood of \((\bar x,0)\).
Hence, the smallest eigenvalue of \(\nabla^2\Phi_z(y)\) remains positive for all \((z,y)\) sufficiently close to \((\bar x,0)\), which proves local strong convexity.
\end{proof}


\subsubsection{Sparse simplex projection problem}\label{example:sparse-simplex}

Consider the following problem
\begin{equation}
\min_{x \in \mathbb{R}^n} 
\left\{ 
\frac{1}{2}\|x - z\|^2 \; \big| \; e^\top x = 1,\; x \ge 0,\; \|x\|_0 \le k
\right\},
\label{eq:sparse-simplex}
\end{equation}
where $z \in \mathbb{R}^n$ is given and $k < n$ is the sparsity level. 
Let
$
\mathcal{K} := \{x \in \mathbb{R}^n \mid x\geq 0,\; \|x\|_0 \le k\},
$
and consider the dual formulation
\begin{equation}
0 \in F(y) := e^\top \Pi_{\mathcal{K}}(z + e y) - 1,
\label{eq:sparse-simplex-dual}
\end{equation}
where $\Pi_{\mathcal{K}}(x)$ can be computed explicitly as
\[
\Pi_{\mathcal{K}}(x) = [(x)^+]_k.
\]
Here $(\cdot)^+$ denotes the componentwise positive part, and $[\cdot]_k$ denotes the vector obtained by keeping the $k$ largest (positive) entries while setting all the others to zero.
Substituting the projection expression into \eqref{eq:sparse-simplex-dual}, we obtain
\[
F(y) = e^\top [(z + e y)^+]_k - 1.
\]
The function $F(y)$ is continuous, piecewise linear, nondecreasing in $y$, and strictly increasing in a neighborhood of any root of \eqref{eq:sparse-simplex-dual}. Hence, \eqref{eq:sparse-simplex} is globally dual representable, and \eqref{eq:sparse-simplex-dual} admits a unique solution at which the dual objective is locally strongly convex. The explicit solution can be derived by sorting the entries of $z$ in descending order, say $z_{(1)} \ge z_{(2)} \ge \dots \ge z_{(n)}$. 
Define
\[
y^{(j)} = \frac{1 - \sum_{i=1}^{j} z_{(i)}^+}{j}, \quad j\in[n].
\]
Let $j^*$ be the largest index in $ \{1,\dots,k\}$ such that
$
 z^+_{(j^*)} + y^{(j^*)} > 0. 
$
Then the optimal $x^*$ is given by
\[
x_i^* = \max\{z_i + y^{(j^*)}, 0\}, \quad i = 1,\dots,n,
\]
with only the $k$ largest nonnegative components remaining nonzero. A similar formula has also been derived from a different perspective in \cite{kyrillidis2013sparse}, while our simple dual representation provides an alternative and unified derivation.

This result can be naturally extended to the matrix setting, where the sparse simplex projection is applied to the eigenvalues of a symmetric matrix. Consider
\begin{equation}
\min_{X \in \mathbb{S}^n}
\left\{
\frac{1}{2} \|X - Z\|^2 \;\big|\;
\operatorname{Tr}(X) = 1,\; X \succeq 0,\; \operatorname{rank}(X) \le k
\right\},
\label{eq:sparse-simplex-matrix}
\end{equation}
which can be interpreted as the projection of a symmetric matrix $Z$ onto the intersection of the spectrahedron $\{X \succeq 0, \operatorname{Tr}(X) = 1\}$ and the nonconvex rank-constrained set. 
This formulation often appears in sparse PCA and low-rank approximation problems. 
Since the objective and constraints are spectral functions, problem \eqref{eq:sparse-simplex-matrix} can be reduced to the vector case through the eigenvalue decomposition $Z = U \operatorname{Diag}(\lambda(Z)) U^\top$. 
Specifically, the projection of $Z$ onto the feasible set can be computed by projecting its eigenvalue vector $\lambda(Z)$ onto the sparse simplex, and then reconstructing the projected matrix via
\[
X^* = U \operatorname{Diag}(\Pi_{\mathcal{K}\cap \E}(\lambda(Z))) U^\top,
\]
where $\E\coloneqq\{x\mid e^\top x=1\}$. This equivalence follows from the spectral invariance of both the Frobenius norm and the constraints, and shows that the matrix projection inherits the structure of the vector projection; see \cite{ding2018spectral} for a general treatment of spectral operators.


\subsubsection{Burer--Monteiro projection}

In many semidefinite programming (SDP) problems, the feasible set takes the form
\[
\mathcal{F}
:= 
\Bigl\{
    X \in \mathbb{S}^n_+
    \;\Big|\;
    \langle A_i, X \rangle = a_i,~ i \in \mathcal{I},\;
    \langle B_j, X \rangle \ge b_j,~ j \in \mathcal{J},\;
    \langle P^\top P, X \rangle = 0
\Bigr\}.
\]
Such constraints are difficult to handle directly because they involve the
intersection of affine equalities and inequalities with the PSD cone, making projection and first-order
optimality analysis highly nontrivial.
To alleviate this difficulty, one often applies the
Burer--Monteiro (BM) factorization
\[
X = R R^\top, \qquad R \in \mathbb{R}^{n\times r},
\]
which replaces the PSD constraint $X\succeq 0$ by a smooth low-rank manifold.
The corresponding feasible region in the factorized variable is
\[
\mathcal{M}
:=
\Bigl\{
    R \in \mathbb{R}^{n\times r}
    \;\Big|\;
    \langle A_i, R R^\top \rangle = a_i,~ i \in \mathcal{I},\;
    \langle B_j, R R^\top \rangle \ge b_j,~ j \in \mathcal{J},\;
    P R = 0
\Bigr\}.
\]
Although the BM reformulation transforms the PSD cone into a smooth manifold,
it introduces nonconvexity into the feasible set, 
which poses new difficulties for projection-type algorithms.
In particular, many existing methods such as projected gradient (PG),
proximal point, or Riemannian optimization schemes~\cite{tang2024solving,wang2025solving,wang2023decomposition} require computing
the Euclidean projection onto~$\mathcal{M}$:
\begin{equation} \label{eq-projM}
\min_{R \in \mathcal{M}} \;\; \frac12 \|R - R_0\|_F^2,
\end{equation}
which often becomes the dominant computational cost in practice.
To derive a dual representation of this projection, 
let us define the intermediate set
\[
\mathcal{K}
:=
\Bigl\{
    R \in \mathbb{R}^{n\times r}
    \;\Big|\;
    \langle A_i, R R^\top \rangle = a_i,~ i \in \mathcal{I},\;
    \langle B_j, R R^\top \rangle \ge b_j,~ j \in \mathcal{J}
\Bigr\},
\]
and note that the constraint $P R = 0$ can be viewed as an affine constraint.
By applying the general dual representation framework introduced earlier,
the projection onto $\mathcal{M}$ can be characterized by the inclusion
\begin{equation}\label{BM-proj}
    0 \in P\,\Pi_{\mathcal{K}}(R_0 + P^\top y).
\end{equation}
Under suitable qualification and regularity assumptions of $\mathcal{M}$ as specified in Theorem~\ref{Thm:nonconvex-dual-repre},
there exists a dual solution $y^*$ of~\eqref{BM-proj}, and 
hence an optimal solution $R^*= \Pi_{\cal K}(R_0 + P^\top y^*)$ can be
found for \eqref{eq-projM}.
A detailed discussion of such a projection formulation for a special
case of ${\cal K}$
and its algorithmic applications
can be found in~\cite{RiNNAL,RiNNAL+,RiNNALPOP}.


\subsection{Proximal gradient problem}\label{subsec-PG}

We consider the following constrained optimization problem:
\begin{equation*}\label{eq:fgA}
\min_{x\in \bX } \; \bigl\{ g(x) + f(x) \;\mid\; Ax=b \bigr\},
\end{equation*}
where $g:\bX \to \R$ is continuously differentiable, $f:\bX \to \Rbar$ is proper and lsc, but possibly nonsmooth and nonconvex, and $A:\bX \to \bY$ is a surjective linear operator.
Given a current iterate $x^k$, the proximal gradient method generates the next iterate by solving the subproblem
\begin{equation}\label{eq:pg_subproblem}
x^{k+1} \in \underset{{x\in \bX}}{\arg\min} \left\{
g(x^k)+\langle 
\nabla g(x^k), x - x^k \rangle
+ \frac{1}{2\lambda_k}\|x-x^k\|^2
+ f(x)
\;\middle|\;
Ax=b
\right\},
\end{equation}
where $\lambda_k>0$ is a stepsize parameter. By completing the square and denoting
\[
z^k \coloneqq x^k - \lambda_k \nabla g(x^k),
\]
the subproblem \eqref{eq:pg_subproblem} can be equivalently rewritten as
\begin{equation}\label{eq:pg_prox_form}
x^{k+1} \in
\operatorname{Prox}_{\lambda_k \fE}
\bigl(z^k\bigr)
:= \underset{x\in \bX}{\arg\min}
\left\{
f(x)
+ \frac{1}{2\lambda_k}\bigl\|x - z^k\bigr\|^2
\;\middle|\;
Ax=b
\right\},
\end{equation}
which is exactly in the form of the general proximal operator \eqref{Prox}. Under the conditions stated in Theorem~\ref{Thm:nonconvex-dual-repre}, the proximal subproblem \eqref{eq:pg_prox_form} is dual-representable and can be computed by solving
\[
0 \in A\,\Prox_{\lambda_k f}\bigl(z^k + A^\top y\bigr) - b.
\]
An illustrative example of the affine-constrained smoothly clipped absolute deviation (SCAD) proximal gradient step is provided in Subsection~\ref{exp:SCAD}.


\subsection{Proximal composite problem}

We now consider the general {composite optimization problem}
\begin{equation}\label{eq:comp}
    \min_{x\in \bX } \; \big\{ f(x) + g(Ax) \big\},
\end{equation}
where \(f:\bX \to\Rbar\) and \(g:\bY\to\Rbar\) are proper, lsc functions, and \(A:\bX \to\bY\) is a surjective linear operator.  
Problem~\eqref{eq:comp} can be equivalently written in the constrained form as
\begin{equation*}\label{eq:comp-cons}
    \min_{x\in\bX ,\,y\in\bY} \; \bigl\{ f(x)+g(y)\;\mid\; Ax - y = 0 \bigr\}.
\end{equation*}
Given proximal centers \((x_0, y_0)\) and penalty parameters \(\lambda > 0\), the corresponding {proximal subproblem} reads
\begin{equation*}\label{eq:prox-sub}
\min_{x\in\bX ,\,y\in\bY}\;
\Bigl\{
f(x)+g(y)
+\frac{1}{2\lambda}\|x-x_0\|^2
+\frac{1}{2\lambda}\|y-y_0\|^2
\;\Bigm|\;
Ax - y = 0
\Bigr\}.
\end{equation*}
Applying the general dual representation developed in Section~\ref{sec:dual}, the dual problem is
\[
0\in\begin{bmatrix}
    A&-I
\end{bmatrix}\Prox_{\lambda(f+g)} \left(\begin{bmatrix}
    x_0\\y_0
\end{bmatrix}+\begin{bmatrix}
    A^\top\\-I
\end{bmatrix}u\right).
\]
Since the inner minimization with respect to \((x,y)\) in the proximal mapping 
\(\Prox_{\lambda(f+g)}\) is separable across the variables, 
the inclusion above can be equivalently decomposed as
\begin{equation}\label{eq:prox-separable}
    0 \in A\,\Prox_{\lambda f}(x_0 + A^\top u)
        - \Prox_{\lambda g}(y_0 - u).
\end{equation}

\begin{example}[\bf Sparse regression problem]
Consider the following problem:
\[
\min_{x \in \bX } \;\Bigl\{\, \|Ax - b\|^2 + \lambda\|x\|_0 \Bigr\},
\]
where $\lambda>0$ is a given regularization parameter.  
Let
\[
f(x) = \|x\|_0, 
\qquad 
g(y) = \tfrac{1}{2}\|y - b\|^2.
\]
Then their proximal mappings admit simple closed forms:
\[
\Prox_{\lambda f}(z)
= H_{\tau}(z),
\quad
\text{where } (H_{\tau}(z))_i =
\begin{cases}
z_i, & |z_i| > \tau, \\[1mm]
\{0,\, z_i\}, & |z_i| = \tau, \\[1mm]
0, & |z_i| < \tau,
\end{cases}
\qquad 
\tau = \sqrt{2\lambda},
\]
and
\[
\Prox_{\lambda g}(w) = \dfrac{w + \lambda b}{1 + \lambda}.
\]

\noindent
Note that the hard-thresholding operator $H_{\tau}$ is set-valued at the threshold points $\{z\mid |z_i| = \tau \;\mbox{for some}\; i\in[n]\}$, corresponding to the non-uniqueness of the proximal mapping of the $\ell_0$ function.
Substituting these expressions into the proximal
subproblem~\eqref{eq:prox-separable} yields the dual inclusion
\[
0 \in A\,H_{\tau}(x_0 + A^\top u)
- \frac{y_0 - u + \lambda b}{1 + \lambda},
\]
which can be equivalently written as
\[
-u \in (1 + \lambda)\,A\,H_{\tau}(x_0 + A^\top u)
- y_0 - \lambda b.
\]
We focus on the branch where $(H_\tau(z))_i$ takes the value $0$ when $|z_i| = \tau$, 
acknowledging that $H_\tau$ is set-valued at the threshold. Under this convention, the hard-thresholding operator $H_{\tau}$ becomes piecewise affine. Consequently, the mapping $u \mapsto (1 + \lambda)A H_{\tau}(x_0 + A^\top u)$ 
is also piecewise affine. Once the active support
\[
S := \{\,i \mid |(x_0 + A^\top u)_i| > \tau\,\}
\]
is determined, the equation above becomes linear:
\[
-\bigl(I + (1+\lambda)A_S A_S^\top\bigr) u
= (1+\lambda)A_S (x_{0})_S - y_0 - \lambda b,
\]
where $A_S$ denotes the submatrix of $A$ formed by columns indexed by $S$.  
Since $\bigl(I + (1+\lambda)A_S A_S^\top\bigr)$ is nonsingular, the unique solution is
\[
u_S^*
= -\bigl(I + (1+\lambda)A_S A_S^\top\bigr)^{-1}
\bigl((1+\lambda)A_S (x_{0})_S - y_0 - \lambda b\bigr),
\]
subject to the consistency conditions
\[
\begin{cases}
|(x_0)_i + (A^\top u_S^*)_i| > \tau, & i \in S, \\[1mm]
|(x_0)_i + (A^\top u_S^*)_i| \leq \tau, & i \notin S.
\end{cases}
\]
\end{example}


\section{Numerical experiments}
\label{sec:numerical}
In this section, we conduct numerical experiments to demonstrate the effectiveness of solving the convex dual problem~\eqref{D} to obtain optimal solutions of the original nonconvex primal problem~\eqref{Prox}.
All experiments are performed using {\sc Matlab} R2023b on a workstation equipped with Intel Xeon E5-2680 (v3) processors and 96GB of RAM.


\paragraph{Baseline solvers.}
A wide variety of first- and second-order methods can be used to solve the unconstrained convex dual problem~\eqref{D}, and different solvers can be flexibly combined to improve efficiency. Here, we select only a few representative algorithms as baselines.
For the dual problem~\eqref{D}, we adopt gradient descent (GD) with Barzilai--Borwein (BB) stepsize, the semismooth Newton (SSN) method, and the limited-memory BFGS method (LBFGS).
For the primal problem~\eqref{Prox}, we implement an ADMM scheme.
For the projection problem~\eqref{eq:proj}, we further compare with alternating projection (AltProj), the SLRA package~\cite{markovsky2014software}, and the NewtonSLRA1/2 methods~\cite{schost2016quadratically}.


\paragraph{Stopping conditions.} 
Based on the optimality characterization~\eqref{eq:grad-dual} for \eqref{Prox}, we define the following relative residuals to measure the accuracy of the computed solution:
\[
\mathrm{R}_{\text{feas}} \coloneqq \frac{\lVert Ax-b\rVert}{1+\lVert b\rVert},\quad
\mathrm{R}_{\text{obj}} \coloneqq \frac{|f_{\lambda,z}(x)-f_{\lambda,z}(x^*)|}{1+|f_{\lambda,z}(x^*)|},\quad
\mathrm{R}_{\text{sol}} \coloneqq \frac{\lVert x-x^*\rVert}{1+\lVert x^*\rVert},
\]
where $x^*$ denotes the optimal solution of~\eqref{Prox}, and 
$f_{\lambda,z}(x) \coloneqq \lambda f(x)+\frac{1}{2}\|x-z\|^2$ denotes the objective function of~\eqref{Prox}.
For the dual approaches, we use $\mathrm{R}_{\mathrm{res}}$ and $\mathrm{R}_{\text{feas}}$ interchangeably, since the feasibility residual $\mathrm{R}_{\text{feas}}$ also provides a certificate of global optimality for any $x \in \operatorname{dom} f$, as guaranteed by Propositions~\ref{prop:SD-eq} and~\ref{prop:basic-props-dual}.
For a given tolerance $\tol>0$, the algorithm is terminated when $x\in\dom f$ and the feasibility residual satisfies $\mathrm{R}_{\text{feas}}<\tol$, or when either the maximum time limit $\timelimit$ or the maximum iteration number $\iterlimit$ is reached. Throughout our experiments, we set $\tol=10^{-6}$ unless otherwise specified, $\timelimit=600 \tt{(secs)}$, and $\iterlimit=1000$ for all solvers.


\paragraph{Implementation.}
For ADMM, we adopt the standard two-block scheme of~\cite{boyd2011distributed} with an adaptive penalty parameter. The penalty parameter $\rho$ is updated according to a heuristic rule with relative ratio $\mu=10$ and scaling factor $\tau=2$. All dual methods are initialized at $y^{0}=0$. For GD, we employ Barzilai--Borwein stepsize combined with a non-monotone line search, whereas both LBFGS and SSN utilize the Armijo line search with Armijo parameter $\gamma=10^{-4}$ and backtracking factor $\sigma=0.5$.


\paragraph{Table notations.} 
We use ``P-'' to denote algorithms for the primal problem~\eqref{Prox} and ``D-'' for those associated with the dual problem~\eqref{D}.
We use ``Iter'' to denote the number of iterations taken by an algorithm.


\subsection{Affine-constrained low-rank projection}

In this subsection, we use the affine-constrained low-rank projection problem~\cite{li2019low,schost2016quadratically,markovsky2014software} to show that, unlike many feasible methods which typically compute only inexact projections, the proposed dual approach is able to compute the exact projection point efficiently. We consider the following affine-constrained low-rank projection problem:
\begin{equation}
\label{eq:lowrank-proj}
\min_{X \in \mathbb{R}^{m \times n}} 
\left\{ 
\frac{1}{2}\|X - Z\|_F^2 
\;\middle|\; 
\mathcal{A}(X) = b,\ \mathrm{rank}(X) \le r
\right\},
\end{equation}
where $Z \in \mathbb{R}^{m \times n}$ is a given observation matrix, $\mathcal{A} : \mathbb{R}^{m \times n} \to \mathbb{R}^p$ is a surjective linear operator, $b \in \mathbb{R}^p$ is a given vector, and $r \le \min\{m,n\}$ is a target rank.
We set $m=n$, fix the rank $r=5$, and randomly generate the ground-truth matrix $\overline{X}=UV^{\top}$, where $U,V\in\mathbb{R}^{n\times r}$ have i.i.d.\ standard Gaussian entries. The affine constraint is imposed by fixing the diagonal entries as $\operatorname{diag}(X)=\operatorname{diag}(\overline{X})$. The observation matrix is then generated as $Z=\overline{X}+ E$,
where $E$ has i.i.d.\ standard normal entries.


\paragraph{Projection accuracy.}
We first compare the solution quality obtained by different solvers under a stringent tolerance $\tol=10^{-14}$ using the solution computed by SSN as the reference. 

\renewcommand{\arraystretch}{1.1}

\begin{longtable}{lcccrr}

\caption{Computational results for affine-constrained low-rank projection with $n=50$.}
\label{tab:SLRA-accuracy}\\

\toprule
{ Method} & { $\mathrm{R}_{\text{feas}}$} & { $\mathrm{R}_{\text{obj}}$} & { $\mathrm{R}_{\text{sol}}$} & { Iter} & { Time (s)} \\
\midrule
\endfirsthead

\toprule
{ Method} & { $\mathrm{R}_{\text{feas}}$} & { $\mathrm{R}_{\text{obj}}$} & { $\mathrm{R}_{\text{sol}}$} & { Iter} & { Time (s)} \\
\midrule
\endhead

\midrule
\multicolumn{6}{r}{Continued on next page}
\endfoot

\bottomrule
\endlastfoot

P-AltProj  & 9.84e-15 & 6.88e-04 & 1.09e-02 & 585 & 0.15 \\
P-NSLRA1   & 1.66e-16 & 8.84e-04 & 1.23e-02 & 5   & 3.78 \\
P-NSLRA2   & 1.09e-16 & 8.84e-04 & 1.23e-02 & 5   & 0.03 \\
P-SLRA     & 2.84e-15 & 7.11e-06 & 1.10e-03 & 27  & 8.16 \\

\midrule

P-ADMM     & 9.31e-15 & 1.70e-15 & 5.58e-14 & 555 & 0.18 \\
D-GD       & 4.40e-15 & 5.73e-15 & 2.66e-14 & 47  & 0.04 \\
D-LBFGS    & 7.76e-15 & 2.12e-15 & 4.60e-14 & 58  & 0.03 \\
D-SSN      & 9.74e-15 & 0.00e+00 & 0.00e+00 & 21  & 0.03 \\

\end{longtable}

As shown in Table~\ref{tab:SLRA-accuracy}, among the primal approaches, only ADMM is able to find the exact projection solution, whereas the other primal solvers exhibit noticeable objective and solution errors despite achieving small feasibility residual. In contrast, all dual solvers consistently achieve high accuracy in both objective value and solution error, together with small feasibility residual. This result highlights the effectiveness of the dual approach for solving~\eqref{eq:lowrank-proj}. Consequently, in the subsequent comparisons, we only compare the dual methods against primal ADMM.


\paragraph{Scalability and efficiency.}
We next evaluate the scalability and efficiency of the dual methods on large-scale instances, in comparison with the primal ADMM solver. To obtain a high-accuracy reference solution, we apply SSN with a stringent tolerance $\tol=10^{-14}$ as the benchmark, while all other methods use the default tolerance $\tol=10^{-6}$.

\renewcommand{\arraystretch}{1.1}

\begin{longtable}{lcccrr}

\caption{Computational results for affine-constrained low-rank projection with $n=1000$.}

\label{tab:SLRA-accuracy-large}\\

\toprule

{ Method} & { $\mathrm{R}_{\text{feas}}$} & { $\mathrm{R}_{\text{obj}}$} & { $\mathrm{R}_{\text{sol}}$} & { Iter} & { Time (s)} \\

\midrule

\endfirsthead

\toprule

{ Method} & { $\mathrm{R}_{\text{feas}}$} & { $\mathrm{R}_{\text{obj}}$} & { $\mathrm{R}_{\text{sol}}$} & { Iter} & { Time (s)} \\

\midrule

\endhead

\midrule

\multicolumn{6}{r}{Continued on next page}

\endfoot

\bottomrule

\endlastfoot

P-ADMM   & 1.28e-05 & 9.67e-06 & 3.78e-04 & 1000 & 238.59 \\
D-GD     & 6.14e-07 & 8.66e-08 & 1.61e-05 & 24   & 6.01   \\
D-LBFGS  & 8.92e-07 & 9.61e-09 & 1.97e-05 & 28   & 7.17   \\
D-SSN    & 2.69e-07 & 3.04e-07 & 3.31e-06 & 3    & 2.42   \\

\midrule

REF-SSN  & 2.15e-15 & 0.00e+00 & 0.00e+00 & 8    & 9.80   \\

\end{longtable}

As shown in Table~\ref{tab:SLRA-accuracy-large}, ADMM attains only modest accuracy after 1000 maximum iterations and over 200 seconds, whereas all dual solvers achieve significantly higher accuracy within a few dozen iterations and under 10 seconds. In particular, SSN converges within three iterations and yields the smallest residuals among all practical solvers. These results demonstrate the superior scalability and efficiency of the proposed dual approach for solving~\eqref{eq:lowrank-proj}. 
Therefore, the result of ADMM is omitted in the subsequent experiments due to its significantly higher runtime compared with the dual methods.


\subsection{Low-rank Euclidean distance matrix projection}

In this subsection, we assess the performance of the proposed dual formulations for the low-rank Euclidean distance matrices (EDMs) projection problem and demonstrate the local linear convergence rate of SSN.
This problem arise in applications such as sensor network localization~\cite{krislock2012euclidean}, molecular conformation~\cite{glunt1993molecular}, and multidimensional scaling~\cite{borg2005modern}. 
A matrix $D \in \mathbb{S}^{n}$ is called an EDM if either of the following equivalent conditions holds:
\begin{enumerate}
[label=\bf {(\roman*)}]
    \item There exists an integer $r \ge 1$ and a set of points $\{x_i\}_{i=1}^n \subset \mathbb{R}^r$ such that
    \begin{equation}\label{EDM-generation}
        D_{ij} = \|x_i - x_j\|_2^2, \quad \forall \ i, j\in [n].
    \end{equation}
    \item $D$ satisfies
    \[
    \diag(D)=0 \quad \text{and} \quad -JDJ \succeq 0,
    \]
    where $J := I - ee^\top/n$ and $e \in \mathbb{R}^n$ denotes the all-ones vector.
\end{enumerate}
In practice, the embedding dimension $r$ for the points 
$\{x_i\}$ is fixed a priori, and the rank of the corresponding matrix $JDJ$ is constrained to be at most $r$~\cite{qi2014computing}. 
Moreover, the distance matrix is only partially known in that a subset of its entries is available as exact measurements, while the remaining observed entries may be contaminated by noise. Let $\Omega \subset \{(i,j)\mid 1\le i<j\le n\}$ denote the index set of the exactly known entries, and let $d_{ij}$ represent the corresponding exact measurements for all $(i,j)\in\Omega$. For the remaining entries, we are given a noisy observation matrix $Z \in \mathbb{S}^{n}$. Consider the projection problem onto the set of low-rank EDMs that are consistent with both the exact structural information on $\Omega$ and the noisy observation $Z$:
\begin{equation}
    \label{prob:EDM}
\min_{D \in \mathcal{K}} 
\left\{ 
\delta_{\mathcal{K}}(D) + \frac{1}{2}\|D - Z\|_F^2 
\;\middle|\; 
\diag(D)=0,\; D_{ij}=d_{ij},\ \forall (i,j)\in\Omega
\right\},
\end{equation}
where the nonconvex set $\K$ is defined as
\[
\mathcal{K}\coloneqq\left\{D\in\S^n\;\mid\; -JDJ \succeq 0,\ \operatorname{rank}(JDJ) \le r\right\}.
\]
The projection onto $\K$ admits the closed-form expression~\cite[Proposition 3.3]{qi2014computing}:
\[
\Pi_{\K}(D)
= -\Pi_{\mathbb{S}_+^n(r)}(-JDJ) + (D - JDJ),
\]
where the projection onto $\mathbb{S}_+^n(r) \coloneqq \{X \in \mathbb{S}_+^n \mid \operatorname{rank}(X) \le r\}$ is given by the truncated eigenvalue decomposition that keeps the $r$ largest nonnegative eigenvalues. 
We generate the observation matrix \(Z\) following the procedure proposed in~\cite{qi2014computing}. The ground-truth EDM \(\overline D\in\mathbb{R}^{n\times n}\) is constructed from 3D helix data \cite{chu2003least,mishra2011low} via~\eqref{EDM-generation}:
\[
x_i=(4\cos(3t_i),\,4\sin(3t_i),\,2t_i),\qquad 
t_i=\frac{2\pi(i-1)}{n-1},\qquad i=1,\dots,n.
\]
A subset \(\Omega\) of off-diagonal entries is sampled uniformly with probability \(1/n\) and treated as exact. Let \(s(\overline D)\) be the empirical standard deviation of the corresponding true distances. Noise is added to the unobserved entries as follows:
\[
z_{ij}=d_{ij}+\sigma \cdot s(\overline D)\cdot \varepsilon_{ij},\qquad (i,j)\notin\Omega,
\]
where \(  \{\varepsilon_{ij}\}\) are i.i.d.\ Gaussian random noise. The observation matrix \(Z\) is then constructed using the exact distances on \(\Omega\) and the noisy distances \(z_{ij}\) on \(\Omega^c\). We set the noise level \(\sigma =10^{-2}\).

\begin{figure}[ht!]
    \centering
    \includegraphics[width=.9\linewidth]{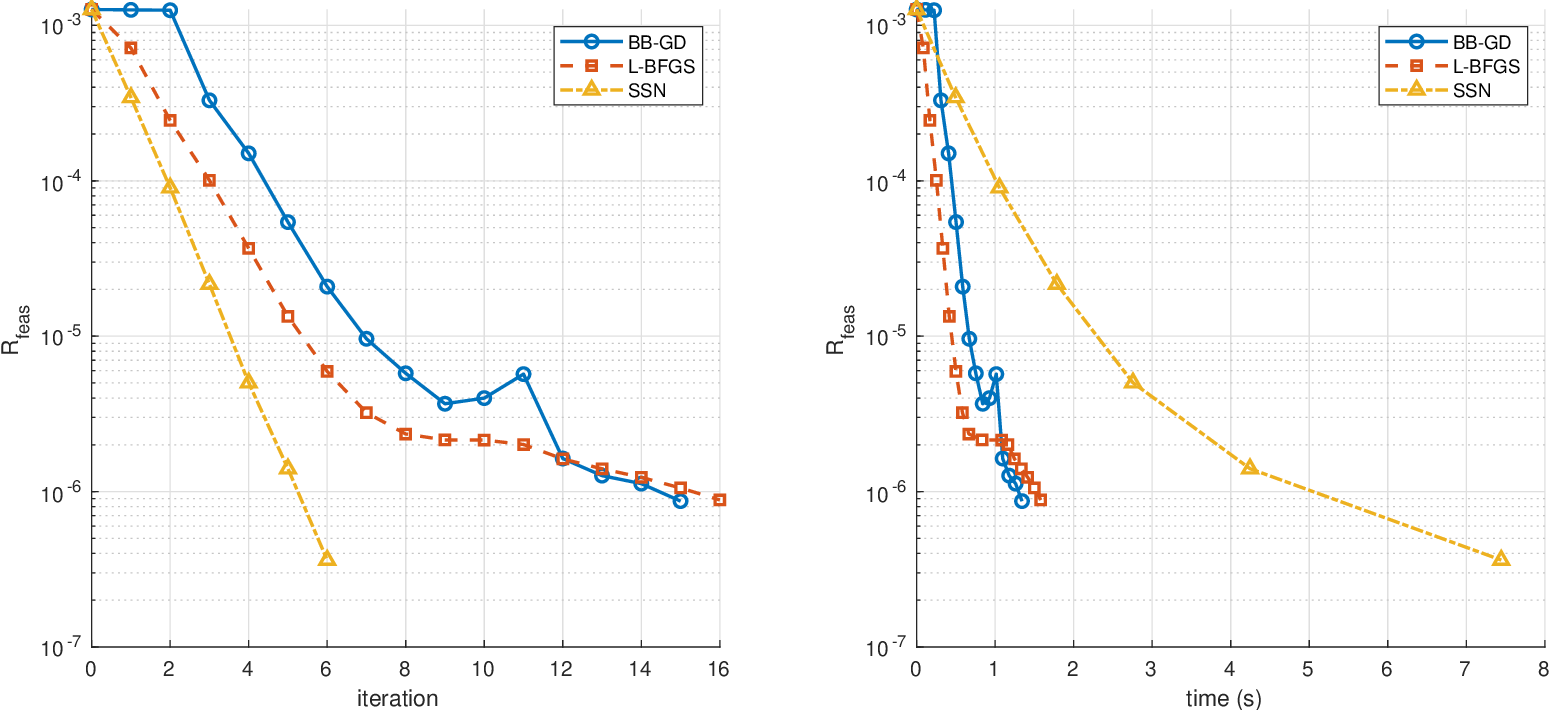}
    \caption{Illustration of convergence rate in terms of $\mathrm{R}_{\text{res}}$ for affine-constrained low-rank EDM projection problem with \(n=1000\).}
    \label{fig:placeholder}
\end{figure}

Figure~\ref{fig:placeholder} illustrates the convergence behavior of the three dual methods for the case \(n=1000\). While all methods exhibit rapid initial decrease in the residual \(R_{\mathrm{res}}\), the SSN method shows a clear linear convergence rate and reaches high accuracy within only a few iterations. In contrast, GD and LBFGS display slower asymptotic convergence despite their fast initial progress. In terms of runtime, GD and LBFGS reduce the residual quickly within the first second, whereas SSN is more expensive per iteration due to its use of second-order information.

\renewcommand{\arraystretch}{1.1}
\begin{longtable}{rrlcrr}
\caption{Computational results for affine-constrained low-rank EDM projection.}
\label{tab:partial-distance-merged}\\
\toprule
{ $n$ } & { $m$ } & { Method} & { $\mathrm{R}_{\text{res}}$} & { Iter} & { Time (s)} \\
\midrule
\endfirsthead

\toprule
{ $n$ } & { $|\mathcal{E}|$ } & { Method} & { $\mathrm{R}_{\text{res}}$} & { Iter} & { Time (s)} \\
\midrule
\endhead

\midrule
\multicolumn{6}{r}{Continued on next page}
\endfoot

\bottomrule
\endlastfoot

1000 & 469  & D-GD    & 8.68e-07 & 15 & 1.48 \\
1000 & 469  & D-LBFGS & 8.84e-07 & 17 & 1.67 \\
1000 & 469  & D-SSN   & 3.63e-07 & 7  & 7.53 \\
\midrule
3000 & 1482 & D-GD    & 8.32e-07 & 13 & 23.20 \\
3000 & 1482 & D-LBFGS & 7.44e-07 & 11 & 21.62 \\
3000 & 1482 & D-SSN   & 6.85e-07 & 7  & 141.04 \\

\end{longtable}

Table~\ref{tab:partial-distance-merged} reports the performance of the three methods for two problem sizes. The SSN method consistently converges in the smallest number of iterations, confirming its fast local convergence. However, its per-iteration cost is significantly higher, leading to much larger total runtime, especially for the larger case with \(n=3000\). In contrast, GD and LBFGS achieve the required accuracy with substantially lower computational time, making them more efficient in large-scale settings.


\subsection{Affine-constrained SCAD proximal mapping}\label{exp:SCAD}

In this subsection, we demonstrate the effectiveness of the proposed dual framework for computing general affine-constrained nonconvex proximal mappings, and highlight its practical potential as a plug-in solver within proximal gradient algorithms.
We investigate the affine-constrained proximal gradient step associated with the smoothly clipped absolute deviation (SCAD) penalty~\cite{fan2001variable}, which serves as a representative example of a nonconvex sparsity-inducing proximal subproblem~\cite{gong2013general} with linear constraints.
We consider the nonconvex SCAD-regularized problem
\begin{equation*}
\label{eq:scad-main}
\min_{x\in \R^n }
\left\{
g(x) + \sum_{i=1}^n \phi_{\mu,a}(x_i)
\;\middle|\;
Ax = b
\right\},
\end{equation*}
where $g:\R^n\to\R$ is continuously differentiable, $A:\R^n\to\R^m$ is surjective, $b\in\R^m$ is given, and $\phi_{\mu,a}$ denotes the SCAD penalty with parameters $\mu>0$ and $a>2$.
At iteration $k$ of a proximal gradient method, denote $z^k:=x^k-\lambda_k\nabla g(x^k)$, then the update $x^{k+1}$ is obtained by solving the affine-constrained proximal subproblem
\begin{equation}
\label{eq:scad-prox-subprob}
x^{k+1}\in \underset{x\in\R^n}{\arg\min}
\left\{
\sum_{i=1}^n \phi_{\mu,a}(x_i)
+\frac{1}{2\lambda_k}\|x-z^k\|^2
\;\middle|\; Ax=b
\right\}.
\end{equation}
The proximal mapping of the unconstrained SCAD is:
\begin{equation*}
\label{eq:scad-prox}
\big(\Prox_{\lambda_k\phi_{\mu,a}}(z)\big)_i=
\begin{cases}
0, & |z_i|\le \mu\lambda_k,\\[2mm]
\dfrac{z_i-\mu\lambda_k\,\mathrm{sign}(z_i)}{1-1/a},
& \mu\lambda_k<|z_i|\le a\mu\lambda_k,\\[3mm]
z_i, & |z_i|>a\mu\lambda_k,
\end{cases}\qquad i=1,\dots,n.
\end{equation*}
Since $\phi_{\mu,a}$ is $(a-1)^{-1}$-weakly convex and hence uniformly prox-regular, under mild regularity conditions, Theorem~\ref{Thm:nonconvex-dual-repre} applies to \eqref{eq:scad-prox-subprob} and guarantees that the affine-constrained proximal step admits a dual characterization through the inclusion
\begin{equation}
\label{eq:scad-dual-inc}
0\in A\,\Prox_{\lambda_k\phi_{\mu,a}}(z^k+A^\top y)-b.
\end{equation}
Since the SCAD proximal operator is piecewise smooth and locally single-valued away from the kink points
\[
\left\{ z \in \mathbb{R}^n \;\big|\; \exists\, i \in [n]
\ \text{s.t.}\ |z_i| \in \{\mu\lambda_k,\; a\mu\lambda_k\} \right\},
\]
the inclusion \eqref{eq:scad-dual-inc} defines a piecewise smooth nonlinear system in the dual variable $y$. For any fixed active regime of the SCAD thresholding, this system reduces to a linear equation, and the primal update
\[
x^{k+1}
=\Prox_{\lambda_k\phi_{\mu,a}}(z^k+A^\top y)
\]
is obtained explicitly. This provides an efficient and globally convergent approach for computing the affine-constrained SCAD proximal step~\eqref{eq:scad-prox-subprob} via its associated dual formulation~\eqref{eq:scad-dual-inc}.
We generate $A \in \mathbb{R}^{m \times n}$ with orthonormal rows by first drawing a matrix $G \in \mathbb{R}^{m \times n}$ with i.i.d.\ Gaussian entries, and then orthonormalizing its rows via a QR factorization. The ground-truth signal $\overline{x} \in \mathbb{R}^n$ is assumed to be $k$-sparse with $k=\lceil \rho n \rceil$, where the support is chosen uniformly at random and the nonzero entries are drawn i.i.d.\ from a Gaussian distribution. We set $b = A \overline{x}$, and the proximal input $z = \overline{x} + \sigma \xi$, where $\xi \sim \mathcal{N}(0,I_n)$. Throughout the experiments, we fix $\mu=1$, $a=3.7$, $\rho = 0.05$, $\sigma=0.01$ and $m=\lceil n/10\rceil$.

\renewcommand{\arraystretch}{1.1}
\begin{longtable}{rlcrr}
\caption{Computational results for affine-constrained SCAD proximal step with $n=10000$.}
\label{tab:SCAD}\\
\toprule
{ $\lambda$ } & { Method } & { $\mathrm{R}_{\text{res}}$ } & { Iter } & { Time (s)} \\
\midrule
\endfirsthead

\toprule
{ $\lambda$ } & { Method } & { $\mathrm{R}_{\text{res}}$ } & { Iter } & { Time (s)} \\
\midrule
\endhead

\midrule
\multicolumn{5}{r}{Continued on next page}\\
\endfoot

\bottomrule
\endlastfoot

0.01 & D-GD    & 6.85e-07 & 12   & 0.09 \\
0.01 & D-LBFGS & 9.57e-07 & 12   & 0.07 \\
0.01 & D-SSN   & 7.07e-07 & 5    & 0.51 \\
\midrule

0.1  & D-GD    & 9.76e-07 & 69   & 0.43 \\
0.1  & D-LBFGS & 4.79e-05 & 1000 & 204.85 \\
0.1  & D-SSN   & 2.73e-07 & 23   & 9.93 \\
\midrule

1    & D-GD    & 9.66e-07 & 212  & 1.45 \\
1    & D-LBFGS & 3.18e-04 & 1000 & 199.88 \\
1    & D-SSN   & 1.06e-01 & 400  & 500.66 \\
\end{longtable}

As shown in Table~\ref{tab:SCAD}, all three dual solvers are able to solve the instances with small  $\lambda$. As $\lambda$ increases, GD remains robust, while LBFGS frequently stagnates and hits the maximum iteration limit, leading to noticeably degraded accuracy. SSN consistently converges in far fewer iterations, however, for large $\lambda$, its performance also deteriorates, reflecting the lack of globalization and its purely local convergence behavior. 
These results confirm that affine-constrained nonconvex proximal steps with SCAD regularization can be efficiently solved through the proposed dual framework.


\section{Conclusions}
\label{sec:conclusion}

In this paper, we establish a unified multiplier-based framework for the analysis and computation of affine-constrained proximal mappings in both convex and nonconvex settings. By identifying dual representability conditions based on subdifferential regularity and prox-regularity, we show that a broad class of affine-constrained nonconvex proximal problems can be equivalently reformulated as unconstrained convex problems in the dual space. This provides a unified treatment of classical convex projection theory and nonconvex proximal analysis within a single variational framework.
From a computational perspective, the resulting dual problem is an unconstrained convex optimization problem whose dimension is equal to the number of affine constraints and is typically much smaller than that of the primal variable.
Moreover, under suitable regularity and nondegeneracy conditions, we show that the dual objective is strongly convex, which guarantees the uniqueness of the dual multiplier and fast local convergence of first- and second-order methods.
As a consequence, the original affine-constrained nonconvex problem can be addressed by applying any suitable unconstrained  optimization algorithm to a convex dual problem. This formulation decouples constraint handling from algorithm design and enables the integration of the framework with a wide range of first- and second-order methods. The framework provides a general basis for treating large-scale structured optimization problems with hard affine constraints, including applications in low-rank optimization, structured sparsity, and general nonconvex constraints.

\bibliographystyle{abbrv}
\bibliography{QCQO}

\end{document}